\documentclass[12pt]{amsart}

\title{Random surfaces with large systoles}

\author[M. Liu]{Mingkun Liu}
\address[Mingkun Liu]{Department of Mathematics, FSTM, University of Luxembourg, L-4364 Esch-sur-Alzette, Luxembourg}
\email{mingkun.liu@uni.lu}

\author[B. Petri]{Bram Petri}
\address[Bram Petri]{Institut de Math\'ematiques de Jussieu--Paris Rive Gauche and  Institut universitaire de France ; Sorbonne Universit\'e and Universit\'e Paris Cit\'e, CNRS, IMJ-PRG, F-75005 Paris, France}
\email{bram.petri@imj-prg.fr}

\date{\today}

\usepackage{amsmath,amssymb,amsthm,graphicx,overpic}
\usepackage{float,enumitem}
\usepackage[dvipsnames]{xcolor}
\usepackage{mathtools}
\usepackage[colorinlistoftodos]{todonotes}
\usepackage{multirow}
\usepackage{pgfplots}
\usepackage{subcaption}

\pgfplotsset{compat=1.7}

\usepackage[cal=euler]{mathalfa}

\usepackage[margin=2.7cm]{geometry}  
\usepackage[indent]{parskip}

\numberwithin{equation}{section}

\usepackage[colorlinks=true,linkcolor=red,citecolor=Green]{hyperref}

\newtheorem{thm}{Theorem}[section]

\newtheorem{prp}[thm]{Proposition}
\newtheorem{cor}[thm]{Corollary}
\newtheorem{lem}[thm]{Lemma}

\newtheorem{innerthmrep}{Theorem}
\newenvironment{thmrep}[1]
  {\renewcommand\theinnerthmrep{#1}\innerthmrep}
  {\endinnerthmrep}

\newtheorem{innercorrep}{Corollary}
\newenvironment{correp}[1]
  {\renewcommand\theinnercorrep{#1}\innercorrep}
  {\endinnercorrep}

\theoremstyle{definition}
\newtheorem{dff}[thm]{Definition}

\newtheorem{rem}[thm]{Remark}

\newcommand{\nc}{\newcommand}
\nc{\dmo}{\DeclareMathOperator}
\usepackage{dsfont}

\nc{\abs}[1]{\left| #1 \right|}
\nc{\bigO}[1]{O\left(#1\right)}
\nc{\card}[1]{\left|#1\right|}
\nc{\ceil}[1]{\left\lceil #1 \right\rceil}
\nc{\CC}{\mathbb{C}}
\nc{\dilog}{\mathcal{L}}
\nc{\floor}[1]{\left\lfloor #1 \right\rfloor}
\nc{\ind}{\mathds{1}}
\nc{\ZZ}{\mathbb{Z}}
\nc{\len}[1]{\left| #1 \right|}
\nc{\littleo}[1]{o\left(#1\right)}
\dmo{\Mat}{Mat}
\nc{\NN}{\mathbb{N}}
\nc{\norm}[1]{\left|\left| #1 \right|\right|}
\nc{\QQ}{\mathbb{Q}}
\nc{\RR}{\mathbb{R}}
\nc{\st}[2]{\left\{\, #1 \,:\, #2\,\right\}}
\dmo{\supp}{supp}
\dmo{\tr}{tr}
\nc{\what}{\widehat}
\dmo{\im}{Im}
\nc{\eps}{\varepsilon}
\dmo{\li}{li}
\dmo{\arccosh}{arccosh}

\dmo{\Tr}{Tr}

\newcommand{\bO}{\mathrm{O}}
\newcommand{\lo}{\mathrm{o}}

\dmo{\area}{area}
\dmo{\conv}{conv}
\dmo{\diam}{diam}
\dmo{\DD}{\mathbb{D}}
\dmo{\dist}{\mathrm{d}}
\nc{\HH}{\mathbb{H}}
\dmo{\Isom}{Isom}
\dmo{\MCG}{MCG}
\dmo{\MPL}{MPL}
\dmo{\Mod}{\mathcal{M}}
\dmo{\PL}{PL}
\nc{\Sphere}{\mathbb{S}}
\dmo{\sys}{sys}
\dmo{\kiss}{Kiss}
\dmo{\Teich}{\mathcal{T}}
\nc{\Torus}{\mathbb{T}}
\dmo{\vol}{vol}
\dmo{\WP}{WP}

\dmo{\convTV}{\;\stackrel{\mathrm{TV}}{\longrightarrow}\;}
\nc{\ExV}[2]{\mathbb{E}_{#1}\left[#2\right]}
\dmo{\EE}{\mathbb{E}}
\nc{\Pro}[2]{\mathbb{P}_{#1}\left[#2\right]}
\dmo{\PP}{\mathbb{P}}
\nc{\distTV}[2]{\mathrm{d}_{\rm TV}\left(#1,#2\right)}
\dmo{\UU}{\mathbb{U}}
\nc{\Var}[2]{\mathbb{V}\mathrm{ar}_{#1}\left[#2\right]}

\dmo{\alt}{\mathfrak{A}}
\dmo{\Aut}{Aut}
\dmo{\Fix}{Fix}
\dmo{\GL}{GL}
\dmo{\Hom}{Hom}
\dmo{\id}{Id}
\dmo{\PGL}{PGL}
\dmo{\PSL}{PSL}
\dmo{\PO}{PO}
\dmo{\Rep}{Rep}
\dmo{\SL}{SL}
\dmo{\SO}{SO}
\dmo{\sym}{\mathfrak{S}}
\dmo{\inv}{\mathcal{I}}
\dmo{\orb}{\mathcal{O}}
\dmo{\stab}{Stab}

\nc{\FF}{\mathbb{F}}

\nc{\calA}{\mathcal{A}}
\nc{\calB}{\mathcal{B}}
\nc{\calC}{\mathcal{C}}
\nc{\calD}{\mathcal{D}}
\nc{\calE}{\mathcal{E}}
\nc{\calF}{\mathcal{F}}
\nc{\calG}{\mathcal{G}}
\nc{\calH}{\mathcal{H}}
\nc{\calI}{\mathcal{I}}
\nc{\calJ}{\mathcal{J}}
\nc{\calK}{\mathcal{K}}
\nc{\calL}{\mathcal{L}}
\nc{\calM}{\mathcal{M}}
\nc{\calN}{\mathcal{N}}
\nc{\calO}{\mathcal{O}}
\nc{\calP}{\mathcal{P}}
\nc{\calQ}{\mathcal{Q}}
\nc{\calR}{\mathcal{R}}
\nc{\calS}{\mathcal{S}}
\nc{\calT}{\mathcal{T}}
\nc{\calU}{\mathcal{U}}
\nc{\calV}{\mathcal{V}}
\nc{\calW}{\mathcal{W}}
\nc{\calX}{\mathcal{X}}
\nc{\calY}{\mathcal{Y}}
\nc{\calZ}{\mathcal{Z}}

\nc{\klav}{Klav\v{z}ar}

\nc{\bi}{\mathbf{i}}
\nc{\bj}{\mathbf{j}}
\nc{\bk}{\mathbf{k}}

\nc{\sfA}{\mathsf{A}}
\nc{\sfG}{\mathsf{G}}
\nc{\sfT}{\mathsf{T}}

\begin{document}

\begin{abstract}
We present two constructions, both inspired by ideas from graph theory, of sequences random surfaces of growing area, whose systoles grow logarithmically as a function of their area. This also allows us to prove a new lower bound on the maximal systole of a closed orientable hyperbolic surface of a given genus.
\end{abstract}

\maketitle

\begin{figure}[ht]
    \centering
    \includegraphics[width=0.75\textwidth]{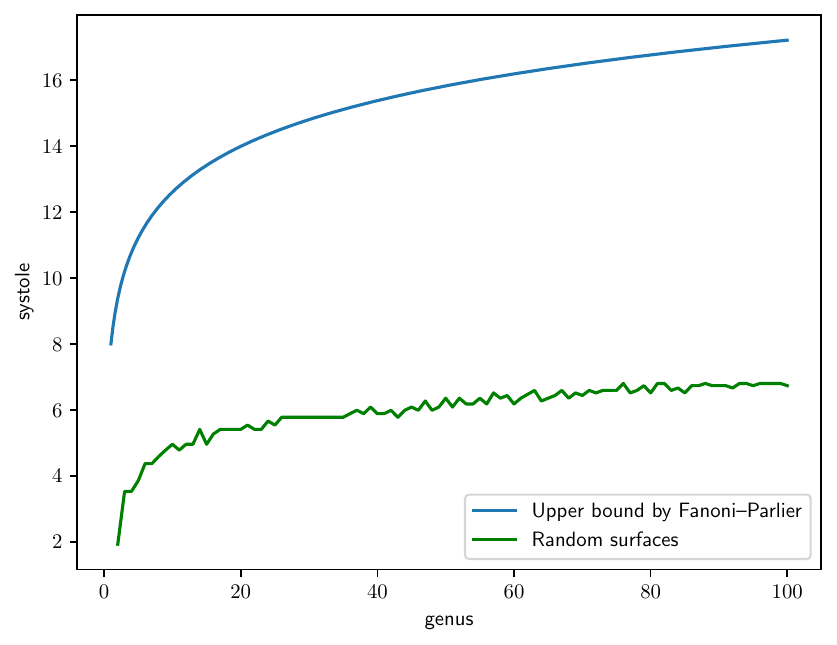}
    \caption{The systole of random Bely\u{\i} surfaces with one cusp, generated with a method inspired by ideas of Linial--Simkin, versus the best known upper bound.}
    \label{fig:examples}
\end{figure}

\section{Introduction}

The systole $\sys(X)$ of hyperbolic surface $X$ is the length of the shortest closed geodesic on $X$. It is known that the function $\sys \colon \calM_{g,n}\to (0,\infty)$, where $\calM_{g,n}$ denotes the moduli space of complete orientable hyperbolic surfaces of genus $g$ with $n$ cusps, admits a maximum \cite{Mumford}. However, what the value of this maximum is and which surfaces realize it, is wide open in general. What is known, is that the maximum is at most logarithmic as a function of $g+n$ \cite{Brooks_injrad,BuserSarnak,KatzSchapsVishne,Schmutz_congruence,FanoniParlier,FBP}. However, even the question of whether the sequence
\[
\frac{\max\st{\sys(X)}{X\in\calM_g}}{\log(g)},\quad g\geq 2
\]
has a limit as $g\to\infty$ is currently open. It is known, due to work by Brooks \cite{Brooks_injrad} and Buser--Sarnak \cite{BuserSarnak} that the limit supremum of this sequence is at least $\frac{4}{3}$. Buser and Sarnak also proved that the limit infimum is positive.  This bound on the limit infimum was recently made effective by Katz--Sabourau \cite{KatzSabourau_newpaper} using a modified version of Buser and Sarnak's construction. Katz and Sabourau obtain $\frac{19}{120}$ as a lower bound. 

The main question in this paper is whether random constructions can provide sequences of surfaces with logarithmic systoles. The interest in doing this, is that the current examples form quite sparse sequences, so we get many more examples and also much more flexible constructions. For instance, one of the consequences of what we do is the following improvement on the result by Katz--Sabourau:
\begin{thm}\label{thm_liminf}
We have
\[
\liminf_{g\to\infty} \frac{\max\st{\sys(X)}{X\in\calM_g}}{\log(g)} \geq \frac{2}{9}.
\]
\end{thm}

What we really prove is that a certain random construction (defined below) can produce a closed surface $X_g$ of genus $g$ for every $g\geq 2$, whose systole satisfies $\mathrm{sys}(X_g)\geq (\frac{2}{9}+\lo(1)) \cdot \log(g)$. This in turn implies the theorem above.

It is known that in the most studied models, a random surface Benjamini--Schramm converges to the hyperbolic plane but has almost surely bounded systole as the area tends to infinity \cite{BM,Pet1,Mir,MirPet,MageePuder,MageeNaudPuder,PuderZimhoni}.
In other words, for random surfaces of large systole (like $X_g$ above), different models are needed. This paper will present two such models. 

Our first model is a hyperbolic analogue of a theorem by Linial--Simkin \cite{LinialSimkin} on random graphs: a version with large systoles of Brooks--Makover's \cite{BM} random Bely\u{\i} surfaces:
\begin{thm}\label{thm_main_linsim}
Given $n\in\NN$ and $c>0$ define the following random process:

\begin{itemize}
\item Let $X_n^{(0)}$ be a surface with boundary, built out of $2n$ ideal hyperbolic triangles, glued together with zero shear. Moreover, suppose that every triangle is incident to two other triangles and $X_n^{(0)}$ contains no closed geodesics of length less than $c\cdot\log(n)$.
\item for $t=1,\ldots,n$, create $X_n^{(t)}$ as follows. Let $\calA^{(t)}$ denote the set of pairs of sides of triangles on $\partial X_n^{(t-1)}$, that when glued together do not create closed geodesics of length less than $c\cdot\log(n)$. If $\calA^{(t)}\neq \varnothing$, pick a pair of sides in $\calA^{(t)}$ uniformly at random and glue them together with an orientation-reversing isometry of zero shear. If $\calA^{(t)} = \varnothing$, set $X_n^{(t)}=X_n^{(t-1)}$.
\end{itemize}

\noindent If $c<\frac{2}{9}$, then with high probability\footnote{Here and throughout the paper, ``with high probability'' means ``with probability tending to $1$ as the underlying parameter ($n$ in this case) tends to infinity.} as $n\to\infty$, the process above saturates, i.e.\ it produces a random hyperbolic surface $X_n\coloneqq X_n^{(n)}$ of area $2\pi n$ without boundary and with systole 
\[
\sys(X_n) \geq c\cdot \log(\area(X_n)).
\]
\end{thm}

We produced Figure~\ref{fig:examples} by running a variant of the algorithm from the theorem above. A database of the surfaces we found is attached to the arXiv version of this paper. For more details, see  Section~\ref{sec_examples}. Theorem \ref{thm_liminf} is based on another variant of this construction, see Section \ref{sec_maxsys}.

Our bound on the constant $c$ is worse than that of Linial--Simkin. This is due to the mix of hyperbolic geometry and combinatorics: a loop of $k$ edges in the graph dual to the triangulation of $X_n$ can be homotopic to a closed geodesic of length $\approx\log(k)$. In particular, if we were to directly apply Linial and Simkin's construction, we would only get a lower bound of the order $\log\log(\area)$ on the systoles of our surfaces. By adapting the argument, we can recover logarithmic growth, but with a worse constant. On top of that, because of similar reasons, our current strategy does not allow us to apply the ``nibbling'' procedure of Linial--Simkin. If this were made to work we estimate it might be possible to bring the bound on the constant up to $\frac{1}{3}$. The surfaces $X_n$ have cusps, but, using \cite[Lemma~3.1]{Brooks}, their conformal compactifications $\overline{X}_n$ are closed hyperbolic surfaces whose systole is still larger than  $(c+\lo(1))\cdot \log(\area(X_n))$.

The second construction we investigate is random \emph{regular} covers. That is, we first fix a hyperbolic surface $X=\Gamma\backslash\HH^2$ and a sequence $(G_n)_n$ of finite groups. Then we take $\varphi_n\in \Hom(\Gamma,G_n)$ uniformly at random and define $X_{G_n} = \ker(\varphi_n)\backslash\HH^2$. By construction, the surface $X_{G_n}$ admits a regular cover $X_{G_n}\to X$. If $\varphi_n$ is surjective, which turns out to be generic as $n\to\infty$ for our choices of sequences finite groups (see Section \ref{sec_covers}), then this cover is of degree $\card{G_n}$ and hence the area of $X_{G_n}$ generically satisfies $\area(X_{G_n}) = \card{G_n}\cdot \area(X)$. This model again is an analogue of a construction in graph theory, this time of the random Cayley graphs studied by Gamburd--Hoory--Shahshahani--Shalev--Vir\'ag \cite{GHSSV}. In particular, this is a different model than that of random covering surfaces studied in \cite{MageePuder,MageeNaudPuder}, the random covers studied in these papers are rarely regular.  We obtain:

\begin{thm}\label{thm_main_covers}
Let $X$ be a complete orientable hyperbolic surface of finite area and let $\mathbf{G}$ be a connected, simply connected, simple algebraic group defined over $\QQ$. Then there exists a constant $c>0$ and a prime $p_0$ such that, with high probability as $\card{\FF} \to \infty$,
\[
    \sys(X_{\mathbf{G}(\FF)})
    \geq
    (c+\lo(1)) \cdot \log(\area(X_{\mathbf{G}(\FF)})),
\]
for any sequence of finite fields $\FF$  whose characteristics all exceed $p_0$. 
\end{thm}

If $X$ is closed, then one can take $c = \frac{1}{\dim(\mathbf{G})}$. Moreover, if $X$ is non-compact, the assumptions that $\mathbf{G}$ is simply connected and that the characteristic of $\FF$ is large enough can be dropped. That is, the results hold for any sequence of finite groups of Lie type of bounded rank whose order tends to infinity (but the constant $c$ we obtain is smaller than $\frac{1}{\dim(\mathbf{G})}$). It would be interesting to generalize the closed case to a larger class of groups (see also Remark~\ref{rem_irreducibility_problem}). In Section~\ref{sec_covers} we will also comment on random regular covers coming from symmetric groups, but the results there are much less conclusive.

Combining the result above with work due to Bourgain--Gamburd \cite[Theorem 3]{BourgainGamburd}, Brooks and Burger \cite{Brooks_transfer,Burger}, Magee \cite{Magee_letter} and Naud \cite{Naud_determinants}, we also obtain bounds on the spectral gap, the diameter and the determinant of the Laplacian. In the corollary below, $\lambda_1(X)$ denotes the first non-zero eigenvalue of the Laplacian of a closed hyperbolic surface $X$:
\begin{cor}\label{cor_spectral_gap}
Let $X$ be a closed orientable hyperbolic surface, then there exist constants $C,\eps>0$ such that, with high probability as $p\to\infty$,
\[
\lambda_1(X_{\SL_2(\FF_p)}) > \eps \quad \text{and} \quad \mathrm{diameter}(X_{\SL_2(\FF_p)})\leq C\cdot \log(\area(X_{\SL_2(\FF_p)})).\]
Moreover, if $\det(\Delta_p)$ denotes the  $\zeta$-regularized determinant of the Laplacian $\Delta_p$ of $X_{\SL_2(\FF_p)}$, then
\[
\frac{\log(\det(\Delta_p))}{\area(X_{\SL_2(\FF_p)})} \longrightarrow E \quad \text{in probability,}
\]
where $E=(4\zeta'(-1)-1/2+\log(2\pi))/4\pi = 0.0538\ldots$.
\end{cor}

The bounds on the spectral gap and the diameter were already known to be true for $\SL_2(\FF_p)$ for a dense sequence of primes $p$, due to work by Breuillard--Gamburd \cite{BreuillardGamburd} (again combined with \cite{Brooks_transfer, Burger, Breuillard} and \cite{Magee_letter}). The reason that we obtain our corollary only for $\SL_2$ is that for other finite groups of Lie type, large girth alone does not suffice for expansion. Instead, one needs to know that short words don't concentrate in proper subgroups of $\mathbf{G}(\FF)$. The reason we only state the result for closed hyperbolic surfaces is that a uniform spectral gap for non-compact surfaces of finite area does not need our Theorem~\ref{thm_main_covers} and instead follows immediately from by combining the work on expansion of random Cayley graphs by Bourgain--Gamburd \cite{BourgainGamburd} and Breuillard--Green--Guralnick--Tao \cite{BGGT} with the Brooks--Burger transfer method\footnote{This method is usually stated for covers of a compact manifold, but because the geometry of cusps of hyperbolic surfaces of finite area is well understood, it can be generalized to this case using ideas similar to those in \cite{Brooks_apollonian, Hamenstaedt}.} \cite{Brooks_transfer, Burger, Breuillard}.

We also note that it follows that the surfaces from Corollary \ref{cor_spectral_gap} and the conformal compactifications of the non-compact surfaces of Theorem~\ref{thm_main_covers} form families of surfaces such that the ratio $\mathrm{diameter}(X)/\mathrm{systole}(X)$ is uniformly bounded. How the minimum of this quantity behaves as the underlying genus tends to infinity is currently open. This question has for instance been evoked in \cite{BCP_mindiam, BalacheffDespreParlier} (analogous questions have also been raised in graph theory \cite{LinialSimkin,ArzBis}).

Finally, it is natural to ask how many non-isometric surfaces both random processes generate. We prove lower bounds using ideas of Liebeck--Shalev \cite{LiebeckShalev2}, Burger--Gelander--Lubotzky--Mozes \cite{BGLM}
and Belolipetsky--Gelander--Lubotzky--Shalev \cite{BGLS}. We obtain:

\pagebreak

\begin{thm} \label{thm:many}
    The following statements hold:
\begin{itemize}
\item[(a)] Fix any constant $c < \frac{2}{9}$. Then there are at least 
\[
n^{n+\lo(n)}
\]
pairwise non-isometric cusped hyperbolic surfaces $X_n$ of area equal to $2\pi n$ and systole at least $c \log(\area(X_n))$ that can be obtained by gluing together $2n$ ideal hyperbolic triangles without shear.
\item[(b)]  Let $X$ be a complete orientable non-arithmetic hyperbolic surface of genus $g$ with $n$ cusps. Then there exist constants $a_X,a_X'>0$ such that $X$ has at least
    \[
        a_X \cdot \frac{\card{\SL(m,\FF)}^{2g+n-3}}{\log(\card{\SL(m,\FF))})}
        \geq
        a_X' \cdot \frac{\card{\FF}^{(m^2-1)\cdot (2g+n-3)}}{(m^2-1)\log(\card{\FF})} 
    \]
pairwise non-isometric $\SL(m,\FF)$-covers whose systoles are at least
\[
(c_{m,X} + \lo(1)) \cdot \log(\card{\SL(m,\FF))}).
\]
for some constant $c_{m,X}>0$ depending on $m$ and $X$ only.
\end{itemize}
\end{thm}

The surfaces produced by Theorem~\ref{thm_main_linsim} can be seen as orbifold covers of the modular curve $\PSL(2, \ZZ) \backslash \HH^2$. The first examples of surfaces of logarithmic systole are also covers of the modular curve, namely principal congruence covers \cite{Brooks_injrad, BuserSarnak, Schmutz_congruence}. It follows Theorem~\ref{thm:many}~(a), combined with Lubotzky's upper bound on the number of congruence subgroups \cite{Lubotzky}, that the overwhelming majority of the surfaces we produce here do not correspond to principal congruence covers.

\subsection*{The multiplicative constants}
We finish this introduction by comparing our results to the best known bounds on systoles and the other constructions available in the literature. Currently, the best known upper bounds on systoles of hyperbolic surfaces of finite area  of genus $g$ with $n$ cusps are
\[
\sys(X) \leq \left\{
\begin{array}{ll}
2\log(g) + 2.409 + \lo(1) \quad \text{as }g\to\infty & \text{if }n=0 \text{ \cite{FBP}} \\[3mm]
\min\left\{2\log(g)+8,2\cosh^{-1}\left(\frac{6g+3n-6}{n}\right)\right\} & \text{if } g\neq 0 \text{ \cite{Schmutz_congruence,FanoniParlier}}\\[3mm]
2\cosh^{-1}\mathopen{}\left(\frac{3n-6}{n}\right)\mathclose{} & \text{otherwise \cite{Schmutz_congruence}}
\end{array}\right.
\]
In particular, the last two bounds imply that a sequence of surfaces of logarithmic systole $X_k$ necessarily satisfies $n(X_k) = \lo(g(X_k)^{1-\eps})$ for some uniform $\eps>0$.

Sequences of surfaces with logarithmic systoles were known to exist, but there are not that many constructions available: the first known examples come from arithmetic constructions \cite{Brooks_injrad, Brooks, BuserSarnak, Schmutz_congruence, KatzSchapsVishne} and still hold the record for the multiplicative constant in front of the logarithm, namely $\frac{4}{3}$. More recently, various constructions inspired by graph theory were found in \cite{PW, Pet2} whose multiplicative constants are $1$ and $\frac{4}{7}$ respectively. The largest we obtain with our random constructions is $\frac{1}{3}$ for random $\SL_2(\FF)$ covers of closed surfaces.

\subsection*{The proofs}
The rest of this article is divided into two sections. The first of these, Section~\ref{sec:BLS}, is about random Bely\u{\i} surfaces. We first prove Theorem~\ref{thm_main_linsim}, rephrased as Theorem~\ref{thm_linsim_rephrased}. Theorem~\ref{thm_liminf} is proved in Section~\ref{sec_maxsys} and Theorem~\ref{thm:many} (a) in Section~\ref{sec_counting_belyi}. Section~\ref{sec_covers} deals with random regular covers. We prove Theorem~\ref{thm_main_covers} separately for closed (Theorem~\ref{thm_randcoverscompact}) and non-compact surfaces of finite area (Theorem~\ref{thm_randcovers_noncompact}). Finally, we prove Theorem~\ref{thm:many} (b) in Section \ref{sec_counting_covers}.

\subsection*{Acknowledgements}
We thank Sebastian Hensel for references on geometric group theory, Marco Maculan for algebraic help, Michael Magee, Fr\'ed\'eric Naud and Doron Puder for the bound we needed for Remark \ref{rem_symmetric_group} and Anna Roig-Sanchis for useful conversations surrounding Lemma~\ref{lem_distortion}. We also thank Nati Linial and Michael Simkin for comments on an earlier draft. Moreover, we thank Hugo Parlier for making us aware of the work by Katz--Sabourau and Stéphane Sabourau for sharing a version of their preprint. We also thank the anonymous referee for their careful reading of this text and their comments. We are also 
grateful for the ``Tremplins nouveaux entrants'' grant from Sorbonne Universit\'e, which allowed ML to visit BP in Paris.
ML is supported by the Luxembourg National Research Fund OPEN grant O19/13865598.

\section{Random Bely\u{\i} surfaces}\label{sec:BLS}
In \cite{BM}, Brooks and Makover introduced a model for random hyperbolic surfaces of large genus.
In this model, $2n$ ideal hyperbolic triangles are glued together at random, using gluings without shear, into an oriented cusped hyperbolic surface.
This surface, which by the uniformization theorem can be seen as a Riemann surface, is then conformally compactified so as to obtain a closed Riemann surface\footnote{This conformal structure is the same as that obtained by gluing equilateral Euclidean triangles together using isometries.}.
The resulting surface will be a Bely\u{\i} surface and typically has genus roughly $n/2$ \cite{Gamburd,ChmutovPittel}. Its systole converges to a real-valued random variable, i.e.\ it is asymptotically almost surely bounded \cite{Pet1}. For instance, the expected systole converges to about $2.48$. 

The goal of this section is is to modify Brooks and Makover's construction, using ideas due to Linial--Simkin \cite{LinialSimkin}. Concretely, we will prove Theorem~\ref{thm_main_linsim}, which we will rephrase as Theorem~\ref{thm_linsim_rephrased} below.

\subsection{Geometry and topology}\label{sec_geomtopbelyi}
First, we describe the geometry of an oriented hyperbolic surface $S$ equipped with a triangulation $\calT$ by ideal hyperbolic triangles without shear\footnote{The word triangulation is used somewhat liberally here: $\calT$ has no vertices and hence consists of edges and faces only.}.
We allow $S$ to have boundary and observe that if $S$ does then each of its boundary components contain at least one cusp. When we say cusp here, we mean an open subset of the surface that contains part of the boundary and is isometric to an open neighborhood of an ideal vertex of an ideal hyperbolic triangle (or equivalently of a region in $\HH^2$ delimited by two geodesics with a common endpoint at infinity).

The orientation of $S$ induces an orientation at each vertex of the dual graph $\calG$ to $\calT$. Here an orientation at a vertex of a graph is a cyclic ordering of the outgoing edges at that vertex.
A graph equipped with such an orientation is sometimes called a ribbon graph or a fat graph.
We will attach a half-edge to $\calG$ for each edge in $\calT$ that lies on the boundary, so that all vertices in $\calG$ are trivalent (see Figure \ref{pic_init}).

If $\gamma \colon [0,1] \to \calG$ is a path, we can associate a word $w(\gamma)$ in the letters ``$L$'' and ``$R$'' to it:
at every vertex of $\calG$ that $\gamma$ traverses it turns either right or left with respect to the orientation on $\calG$.
So if we record an ``$L$'' or an ``$R$'' respectively at each turn and put these letters in the correct order, we obtain the word $w(\gamma)$.
We will say that $\gamma$ carries the word $w(\gamma)$. In what follows, we will often forget about the parametrization of a curve $\gamma$, but still speak of ``the'' associated word $w(\gamma)$. Technically, this makes no sense, because the word depends on the choice of parametrization, however, in all cases below, this is of no influence. That is, the resulting definitions are independent of the choice of parametrization.

There will be two types of closed curves on $S$ that are of particular interest to us:
\begin{itemize}
    \item Loops that go around a cusp once. Note that such a loop exists if and only if the cusp in question does not lie on the boundary.
        Depending on their orientation, they can be homotoped into left-hand-turn cycles or right-hand-turn cycles in $\calG$. We note that in $\calG$, these cycles are not necessarily simple.
        The number of cusps in the interior of $S$ hence equals the number of primitive\footnote{Primitive here means that the element of $\pi_1(S)$ induced by the cycle is not a non-trivial power.} left-hand turn cycles in $\calG$.
        So, if $S$ is connected and has no boundary, one computes using the Euler characteristic that
        \[
            \mathrm{genus}(S)
            =
            \frac{n+2-\mathrm{LHT}(\calG)}{2},
        \]
        where $\mathrm{LHT}(\calG)$ is the number of primitive left-hand turn cycles in $\calG$. 
    \item Essential closed curves: closed curves in $S$ that are not homotopic to a point or a  cusp\footnote{Note that we allow these curves to be homotopic into the boundary of $S$.
        This is because the only surfaces with boundary we will see in what follows appear as a step in a process in which the boundary will eventually disappear.}. By classical hyperbolic geometry, these curves, considered up to free homotopy, correspond one-to-one to closed geodesics in $S$. 

        Such a curve $\gamma$ can be freely homotoped to a unique geodesic $\overline{\gamma}$ in $\calG$.
        Here a geodesic in $\calG$ is a path that does not backtrack.
        The length of the unique geodesic $\what{\gamma}$ in $S$ homotopic to $\gamma$ can now be computed using $w(\overline{\gamma})$.
        First of all, replace the letters in $w(\overline{\gamma})$ with the matrices
        \[
            L
            \coloneqq
            \left(
                \begin{array}{cc}
                    1 & 1 \\
                    0 & 1
                \end{array}
            \right)
            \quad
            \text{and}
            \quad
            R
            \coloneqq
            \left(
                \begin{array}{cc}
                    1 & 0 \\
                    1 & 1
                \end{array}
            \right)
        \]
        so that $w(\overline{\gamma})$ becomes a matrix.
        The length of $\overline{\gamma}$ can now be computed by
    \[
        \ell(\what{\gamma})
        =
        2\cosh^{-1} \mathopen{}\left( \frac{\tr(w(\overline{\gamma}))}{2} \right)\mathclose{}.
    \]	  
    The formula above makes sense, because the choice of parametrization of $\gamma$ is of no influence on $\tr(w(\overline{\gamma}))$.
\end{itemize}

\subsection{Properties of words and their traces}
This section records some properties of words in $L$ and $R$ that we shall need below. In what follows we will write $W$ for the set of words in $L$ and $R$. We will need the following elementary lemmas:

\begin{lem}\label{lem_traceincrease}
    Let $w\in W$, then
    \[
        \tr(Lw) = \tr(wL) \geq \tr(w)
        \quad \text{and} \quad
        \tr(Rw) = \tr(wR) \geq \tr(w)
    \]
    and these inequalities are strict when $\tr(w) \neq 2$.
\end{lem}

\begin{lem}\label{lem_tr_lower_bound}
    Let $w\in W$ be a word of $k$ letters with $\tr(w) >2$, then
    \[
        \tr(w) \geq k+1.
    \]
\end{lem}

For proofs of these two lemmas, see \cite[Lemma 3.1]{PW}. Moreover, we need the following counting result:
\begin{prp}[\text{\cite[Proposition~3.4]{PW}}]\label{prp_counting} We have as $m \to \infty$,
\[
\card{\st{w\in W}{ 3\leq \tr(w) \leq m}} = \bO(m^2 \log(m)).
\]
\end{prp}

We remark that the lower bound (of $3$) on the trace is necessary, otherwise the set above would be infinite due to words of the form $L^k$ and $R^k$.

\subsection{The construction}\label{sec_theconstruction}
Now we get to the construction. As Theorem~\ref{thm_main_linsim} states, the goal is to start with a surface with a ``$2$-valent'' triangulation that has no short closed geodesics (see Figure~\ref{pic_init} for an example) and then complete this into a surface that still doesn't have any short closed geodesics. There is a way to do this completion carefully, similarly to how this was done for graphs by Erd\H{o}s--Sachs in \cite{ES}, which was described in \cite{PW}.
The moral of what follows (and also of Linial and Simkin's result in that case of graphs \cite{LinialSimkin}) is that in fact, if we set the obvious condition that no short closed curves are created in any step and then glue at random, the probability that it works is very high.

\begin{figure}[t]
\includegraphics[scale=1]{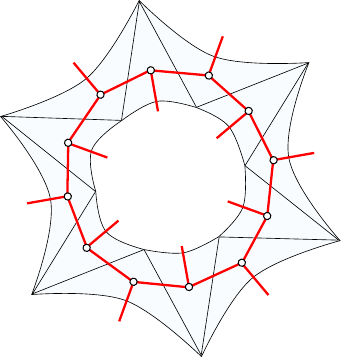}
\caption{An example of an initial configuration $X_n^{(0)}$ (with the orientation induced by the page) and its dual graph. The only primitive cycle carries the word $(LR)^n$, which has trace $\left(\frac{3+\sqrt{5}}{2}\right)^n + \left(\frac{3-\sqrt{5}}{2}\right)^n$ and hence $X_n^{(0)}$ can be used as input for Theorem~\ref{thm_linsim_rephrased}.}
\label{pic_init}
\end{figure}

In order to formally define the conditioning, we need a notion of distance that we can use to measure when two half-edges in an oriented graph, corresponding to two unmatched sides on the boundary of a triangulated surface, are too close to glue together. The material in Section \ref{sec_geomtopbelyi} allows us to associate a word and hence a trace to a path in $\calG$ that connects two half-edges coming out the boundary $\partial S$ of $S$: given two half-edges $h_1, h_2$ of $\calG$ corresponding to two edges of $\calT$ that lie on $\partial S$, we denote by $P_{\calG}(h_1, h_2)$ the set of non-backtracking paths in $\calG$ that start on $h_1$ and end on $h_2$. We now define the trace distance between the half-edges $h_1$ and $h_2$ to be
\begin{equation}\label{eq:tau}
    \tau_{\calG}(h_1, h_2)
    \coloneqq
    \min \big\{ \tr(w(\gamma)) : \gamma \in P_{\calG}(h_1, h_2),\ \tr(w(\gamma)) > 2 \big\}
\end{equation}

The reason for defining this trace distance is that any path $\gamma \in P_{\calG}(h_1, h_2)$ will close up if we glue $h_1$ and $h_2$ together, thus creating a loop whose geodesic representative has trace $\tr(w(\gamma))$. 

\medskip
\fbox{ \begin{minipage}{14cm}
\textbf{The process}: 

\medskip
\underline{Initialisation}: \\
Objects: $\tau_0,n\in\NN$, $t=0$ and
\begin{itemize}
\item[-] a hyperbolic surface $X_n^{(0)}$ with boundary such that
\begin{itemize}
\item each boundary component contains at least one cusp,
\item all the cusps lie on the boundary and
\item all closed geodesics have length at least $\sigma_0 \coloneqq 2\cosh^{-1}(\tau_0/2)$,
\end{itemize}
\item[-] an ideal triangulation $\calT^{(0)}$ with shear $0$ of $X_n^{(0)}$, containing $2n$ triangles, each adjacent to exactly two triangles,
\item[-] the dual graph $\calG^{(0)}$ to $\calT^{(0)}$ whose set of half-edges will be denoted $\calH^{(0)}$,
\item[-] the set of \emph{available} half-edges
\[
\calA^{(0)} = \left\{ \{h_1, h_2\} : h_1, h_2, \in \calH_0, \ \tau_{\calG^{(0)}}(h_1,h_2) \geq \tau_0 \right\}
\]
at time $0$.
\end{itemize}

\medskip
\underline{Iteration}:
Stop if $\calA^{(t)} = \varnothing$.
While $\calA^{(t)}\neq \varnothing$, perform step $t+1$:
\begin{enumerate}
\item Select a pair of half-edges $(h_1,h_2)\in\calA^{(t)}$ uniformly at random.
\item Identify $h_1$ with $h_2$ and identify the corresponding sides in the triangulation $\calT^{(t)}$ so as to create a new surface $X_n^{(t+1)}$ with a triangulation $\calT^{(t+1)}$, whose dual graph is denoted $\calG^{(t+1)}$, the half-edges of which are denoted $\calH^{(t+1)}$.
\item Set 
\[
\calA^{(t+1)} = \left\{ \{h_1, h_2\} : h_1, h_2, \in \calH^{(t+1)}, \ \tau_{\calG^{(t+1)}}(h_1,h_2) \geq \tau_0 \right\}
\]
the set of available half-edges at time $t+1$.
\item Add $1$ to $t$.
\end{enumerate}
\medskip
\underline{Output}: $t_{\max}\in \{0,\ldots,n\}$, the time at which the process stopped.
\end{minipage}}
\medskip

Observe that for all $t=0, 1, \dots, t_{\max}$ we have $|\calH^{(t)}| = 2n - 2t$ and hence if $t_{\max} = n$, then $X_n^{(n)}$ is a surface without boundary and with systole $\geq \sigma_0 = 2\cosh^{-1}(\tau_0/2)$. Our goal will be to prove:
\begin{thm}\label{thm_linsim_rephrased}
    Let $c<\frac{2}{9}$ and $\tau_0=\tau_0(n) \leq n^{c/2}$ for $n\geq 2$.
    Then
    \[
        \lim_{n \to \infty} \PP(t_{\max} = n)
        =
        1.
    \]
\end{thm}
In other words, the process saturates with high probability when $n \to \infty$ and produces a surface $X_n = X_n^{(n)}$ of systole
\[
    \sys(X_n) \gtrsim c\cdot \log(\area(X_n))
\]
as $n \to \infty$.

\subsection{The proof of Theorem~\ref{thm_linsim_rephrased}}
The strategy of proof is the same as that of Linial and Simkin \cite{LinialSimkin}. However, mostly because the trace  of a word cannot be uniformly compared to the number of letters in it, we will need to modify it. For readers who are familiar with Linial and Simkin's proof, we have added comments in this proof to compare the two. However, the proof presented here is self-contained and readers who are not familiar with Linial and Simkin's paper can ignore these comments.

From now on we will suppose that $\tau_0 \leq n^{c/2}$. We moreover fix $a > 0$ such that
\begin{equation}\label{eq:ca}
    a+c<1 \quad\text{and}\quad a>c/2.
\end{equation}
Because $c<\frac{2}{9}$, such an $a>0$ can always be found. We now set
\[
    T
    \coloneqq
    n - \lceil n^{a+c} \rceil.
\]
The reason that our conditions are slightly different than those of Linial and Simkin will become apparent below. We now first state a deterministic lemma that is the analogue of \cite[Lemma 2.3]{LinialSimkin}:

\begin{lem}\label{lem:2.3}
    There exists $n_0 = n_0(c)$ such that for all $n \geq n_0$ and $t \leq T$, the following holds: 
    \begin{itemize}
        \item[i.] $\calA^{(t)} \neq \varnothing$, and hence $t_{\max} > T$. 

        \item[ii.] For every half-edge $h \in \calH^{(t)}$, there exist at least $|\calH^{(t)}| - \bO( n^c \log n)$ half-edges $h'\in \calH^{(t)}$ such that $(h,h') \in \calA^{(t)}$.

        \item[iii.] We have
            \[
                |\calA^{(t)}|
                \geq
                \frac{|\calH^{(t)}|^2}{2} \left( 1 - \frac{\bO(n^c \log n)}{|\calH^{(t)}|} \right).
            \]
    \end{itemize}
\end{lem}
\begin{proof}
    Let $h \in \calH^{(t)}$.
    Proposition~\ref{prp_counting} implies that there are at most $\bO(n^c\log n)$ half-edges that can be reached by a path of trace at most $\tau_0$ that starts from $h$. This implies ii.
    The condition that $t\leq T$ implies that
    \[
        |\calH^{(t)}|
        \geq
        2n - 2T
        \geq
        n^{a+c}        
        >
        0.
    \]
    Combining this with ii, shows that
    \[
        |\calA^{(t)}|
        \geq
        \frac{1}{2} |\calH^{(t)}| \big( |\calH^{(t)}| - \bO(n^c \log n) \big)
    \]
    and hence that $\calA^{(t)} \neq \varnothing$.
\end{proof}

Recall that $\calT$ is an ideal triangulation of the hyperbolic surface $S$ with zero sheer along its edges.
We say that the dual graph $\calG$ of
$\calT$ is \emph{safe} if the following two conditions hold:
\begin{enumerate}
\item for all pairs of half-edges $h_1, h_2$ of $\calG$, $\tau_{\calG}(h_1, h_2) \geq \tau_0$. In other words, all pairs of half-edges are available and
\item for any cusp on the boundary of the associated surface, the number of triangles incident to the cusp is at least $\sqrt{\tau_0-2}$. In other words, if a pair of half-edges $h,h'$ of $\calG$ can be connected by a word of the form $L^k$ or $R^k$, then $k\geq \sqrt{\tau_0-2}$.
\end{enumerate}

Recall that $\calG^{(t)}$ is the dual graph of the triangulation $\calT^{(t)}$ obtained from $\calT^{(0)}$ after $t$ gluings.
The following lemma is the analogue of \cite[Lemma~2.5]{LinialSimkin}:
\begin{lem}\label{lem:safe}
    If $\calG^{(t)}$ is safe for some $t$, then $\calG^{(t+1)}$ is also safe with certainty, hence the process saturates with certainty.
\end{lem}
\begin{proof}
    We shall prove (by contradiction) that if $\calG^{(t)}$ is safe, then $\calG^{(t+1)}$ is also safe, which implies the assertion.

    Suppose that $\calG^{(t)}$ is safe but $\calG^{(t+1)}$ is not, i.e.,
    \begin{itemize}
    \item either there exist $h_1, h_2 \in \calH^{(t+1)}$ such that $\tau_{\calG^{(t+1)}}(h_1, h_2) < \tau_0$.
    \item or there is a left hand turn path going between two half-edges of $\calG^{(t+1)}$ whose combinatorial length (the number of vertices it goes through) is below $\sqrt{\tau_0-2}$.
    \end{itemize}
    Let $\gamma$ be a path in $\calG^{(t+1)}$ that either realizes one of these two inequalities.
    
    The hypothesis that $\calG^{(t)}$ is safe implies that $\gamma$ passes through two half-edges $h_1', h_2' \in \calH^{(t)}$ such that $h_1'$ and $h_2'$ are connected at step $t$, and if we break $\gamma$ at the edge $\{h_1', h_2'\}$, we get two subpaths of $\gamma$: $\alpha$ from $h_1$ to $h_1'$, and $\beta$ from $h_2$ to $h_2'$.
    Both $\alpha$ and $\beta$ are contained in $\calG^{(t)}$.

    First suppose $\gamma$ is a left hand turn path, then both $\alpha$ and $\beta$ are left hand turn paths. Their combinatorial lengths are less than that of $\gamma$, 
     which contradicts the assumption that $\calG^{(t)}$ is safe.
    
    If $\gamma$ is not, then if at least one of $\alpha$ and $\beta$ contains both an $L$ and an $R$, say $\alpha$, then by Lemma~\ref{lem_traceincrease} the trace length of $\alpha$ is less than that of $\gamma$, a contradiction. 
    Finally, without loss of generality, suppose $\alpha$ carries the word $L^n$ and $\beta$ carries the word $R^m$ and $n \leq m$. But then
    \[
    n^2+2\leq mn+2 = \tr(L^n R^m)= \tr(w(\gamma)) = \tau_{\calG^{(t+1)}}(h_1,h_2)<\tau_0,
   \] which means that the cusp corresponding to $\alpha$ on the boundary of the surface corresponding to $\calG^{(t)}$ has two few triangles in it, which is inconsistent with the safety of $\calG^{(t)}$.
\end{proof}

It will be convenient to -- like Linial and Simkin -- use the complete graph $K_{2n}$ on the vertex set of $\calG^{(0)}$, as a graph that contains all the paths that potentially appear in our dual graphs $\calG^{(t)}$ and divide the edges of $K_{2n}$ into two types: chords and the edges of $\calG^{(0)}$:

\begin{dff}
    A \emph{chord} is an edge of $K_{2n}$ that is not an edge of $\calG^{(0)}$.
\end{dff}

Recall that the $t$-th step of the process (also referred to as the operation at time $t$) is carried out on $\calG^{(t)}$, and the outcome is $\calG^{(t+1)}$. Recall also that $\calH^{(t)}$ is the set of half-edges of $\calG^{(t)}$.
$T := n - \lceil n^{a+c} \rceil$ where $c < 2/9$, $a + c < 1$ and $a > c/2$. The following lemma corresponds to \cite[Claim 2.6]{LinialSimkin}, and is the reason that our constants are worse.
\begin{lem}\label{lem:2.6}
    Let
    \begin{itemize}
        \item
            $k$ be an integer such that $k = \bO(\tau_0) = \bO(n^{c/2})$,

        \item
            $s_1,\dots,s_k$ be distinct chords,

        \item
            $0 \leq t_1, \dots, t_k \leq T$ be integers,

        \item
            $U \subset \calH^{(0)}$ be a set of half-edges of size $\card{U} = \bO(\tau_0) = \bO(n^{c/2})$,

        \item
            $E$ be the event where $U \subset \calH^{(T+1)}$, and for each $0 \leq i \leq k$, the chord $s_i$ is selected at time $t_i$.
    \end{itemize}
    
    Then, when $n \to \infty$,
    \[
        \PP(E)
        \leq
        \left( \frac{1}{2n^2} \right)^k \left( 1 - \frac{T}{n} \right)^{\card{U}} (1 + \lo(1)).
    \]
\end{lem}
\begin{proof}
    For convenience let us start by setting some notation.
    Let $t \in \ZZ_{\geq 0}$.
    We denote by
    \begin{itemize}
        \item
            $S_{>t} \coloneqq \{s_i : t_i > t\}$ the set of chords that are to be chosen after time $t$ for the event $E$ to occur,

        \item
            $U_t \coloneqq U \cup \{u : u \in s \text{ for some }s\in S_{>t}\}$ the set of half-edges to be avoided at time $t$ for the event $E$ to occur,

        \item
            $B_t$ the event that
            \begin{itemize}
            \item[-] if $t=t_i$ for some $i$, then $s_i$ is chosen at time $t$; 
            \item[-] otherwise, the chord picked at time $t$  is not incident to $U_t$.
            \end{itemize}
    \end{itemize}
    With this notation in place, we have
    \[
        E = \bigcap_{0 \leq t \leq T}B_t
    \]
    and hence
    \[
        \PP(E)
        =
        \PP(B_0) \prod_{0 \leq t \leq T} \PP(B_{t} \mid B_{t-1}, \dots, B_0).
    \]
    At time $t=t_i$, the chord $s_i$ should be picked among $|\calA^{(t)}|$ available edges.
    By Lemma~\ref{lem:2.3},
    \[
        |\calA^{(t)}|
        \geq
        \frac{|\calH^{(t)}|^2}{2} \left( 1 - \frac{\bO(n^c\log n)}{|\calH^{(t)}|} \right),
    \]
    and therefore, for $1 \leq i \leq k$,
    \begin{equation}\label{eq:ti}
        \PP(B_{t_i} \mid B_{t_i-1},\dots,B_0)
        = \frac{1}{\card{\calA^{(t_i)}}}
        \leq 
        \frac{2}{|\calH^{(t_i)}|^2} \left( 1 + \frac{\bO(n^c \log n)}{| \calH^{(t_i)} |} \right).
    \end{equation}

    At time $t \in \{ 0, \dots, T \} \smallsetminus \{ t_1, \dots, t_k \}$, all chords that touch $U_t$ should be avoided.
    Lemma~\ref{lem:2.3}.ii implies that every half-edge in $U_t$ is contained in at least $|\calH^{(t)}| - \bO(n^c\log n)$ available edges in $K_{2n}$.
    Therefore, $U_t$ is incident to at least
    \[
        \card{U_t} (|\calH^{(t)}| - \bO(n^c \log n)) - \binom{\card{U_t}}{2}
        =
        |U_t| |\calH^{(t)}| \left( 1 - \frac{\bO(n^c \log n)}{|\calH^{(t)}|} \right)
    \]
    available edges (we have used the fact that $|U_{t}| \leq |U| + k = \bO(n^{c/2})$ for all $0 \leq t \leq T $).
    Hence, for any $t \in \{ 0, \dots, T \} \smallsetminus \{ t_1, \dots, t_k \}$, we have
    \begin{equation}\label{eq:neqti}
        \begin{split}
            \PP(B_t \mid B_{t-1}, \dots, B_0)
            & \leq
            1 - \frac{|U_t| |\calH^{(t)}|}{|\calA^{(t)}|}\left( 1 - \frac{\bO(n^c \log n)}{|\calH^{(t)}|} \right) \\
            & \leq
            1 - \frac{2|U_t|}{|\calH^{(t)}|} + \frac{\bO(n^c \log n) |U_t|}{|\calH^{(t)}|^2}
        \end{split}
    \end{equation}
    where we have used $|\calA^{(t)}| \leq \binom{|\calH^{(t)}|}{2} \leq \frac{|\calH^{(t)}|^2}{2}$.
    We claim that
\begin{equation}\label{eq:PE1}
        \PP(E)
        \leq
        \left( \prod_{i=1}^{k} \frac{2}{|\calH^{(t_i)}|^2} \right)
        \Bigg( \prod_{t = 0}^{T} \left( 1 - \frac{2|U_t|}{|\calH^{(t)}|} + \frac{\bO(n^c \log n) |U_t|}{|\calH^{(t)}|^2} \right) \Bigg)
        (1 + \lo(1)).
    \end{equation}
    This estimate follows from \eqref{eq:ti}, \eqref{eq:neqti}, and the following two facts.
    First,
    \begin{multline}\label{eq:ti_o}
        \prod_{i=1}^{k} \left( 1 + \frac{\bO(n^c \log n)}{| \calH^{(t_i)} |} \right)
        =
        \exp \mathopen{}\left( \sum_{i=1}^{k} \log \mathopen{}\left( 1 + \frac{\bO(n^c \log n)}{| \calH^{(t_i)} |} \right)\mathclose{} \right)\mathclose{} \\
        \leq
        \exp \mathopen{}\left( \sum_{i=1}^{k} \frac{\bO(n^c \log n)}{| \calH^{(t_i)} |} \right)\mathclose{}
        \leq
        \exp( k \, \bO(n^{-a} \log n))
        =
        \exp(\bO(n^{c/2-a} \log n))
        =
        1 + \lo(1),
    \end{multline}
    where we need the assumption $a > c/2$ which is stronger than that in \cite{LinialSimkin};
    second,
    \begin{equation}\label{eq:neqti_o}
        \prod_{i=1}^{k}
        \left( 1 - \frac{2|U_{t_i}|}{|\calH^{(t_i)}|} + \frac{\bO(n^c \log n) |U_{t_i}|}{|\calH^{(t_i)}|^2} \right)
        \geq
        1 + \lo(1),
    \end{equation}
    which can be deduced in a similar way.

    For the second product in \eqref{eq:PE1}, we have
    \begin{multline*}
        \prod_{t=0}^{T}
        \left( 1 - \frac{2|U_t|}{|\calH^{(t)}|} + \frac{\bO(n^c \log n) |U_t|}{|\calH^{(t)}|^2} \right)
        =
        \exp \mathopen{}\left( \sum_{t=0}^{T} \log \mathopen{}\left( 1 - \frac{2|U_t|}{|\calH^{(t)}|} + \frac{\bO(n^c \log n) |U_t|}{|\calH^{(t)}|^2} \right)\mathclose{} \right)\mathclose{} \\
        \leq
        \exp \mathopen{}\left( - \sum_{t=0}^{T} \frac{2 |U_t|}{|\calH^{(t)}|} \right)\mathclose{} \exp \mathopen{}\left( \sum_{t=0}^{T} \frac{\bO(n^c \log n) |U_t|}{|\calH^{(t)}|^2} \right)\mathclose{}
    \end{multline*}
    where
    \begin{multline*}
        \sum_{t = 0}^{T} \frac{\bO(n^c \log n) |U_t|}{|\calH^{(t)}|^2}
        =
        \bO(n^{3c/2} \log n) \sum_{t=0}^{T} \frac{1}{|\calH^{(t)}|^2} \\
        \leq
        \bO(n^{3c/2} \log n) \sum_{i=n-T}^{n} \frac{1}{i^2}
        =
        \bO(n^{3c/2-(a+c)} \log n)
        =
        \bO(n^{-a+c/2} \log n)
        =
        \lo(1)
    \end{multline*}
    where once again we need the assumption $a > c/2$.
    Thus,
    \begin{equation}\label{eq:PE2}
        \prod_{t=0}^{T}
        \left( 1 - \frac{2|U_t|}{|\calH^{(t)}|} + \frac{\bO(n^c \log n) |U_t|}{|\calH^{(t)}|^2} \right)
        =
        \exp \mathopen{}\left( - \sum_{t=0}^{T} \frac{2 |U_t|}{|\calH^{(t)}|} \right)\mathclose{} (1+\lo(1)).
    \end{equation}
    As $|U_t|$ can be written as $|U| + 2|S_{>t}|$, we have
    \begin{equation}\label{eq:PE3}
        \sum_{t=0}^{T} \frac{2|U_t|}{|\calH^{(t)}|}
        =
        |U| \sum_{t=0}^{T} \frac{1}{n-t} + \sum_{i=1}^{k} \sum_{t=0}^{t_i-1} \frac{2}{n-t}
        \geq
        |U| \log \frac{n}{n-T} + 2\sum_{i=1}^{k} \log \frac{n}{n-t_i+1}.
    \end{equation}
    Now, combining \eqref{eq:PE1}, \eqref{eq:PE2}, and \eqref{eq:PE3}, we obtain
    \begin{multline*}
        \PP(E)
        \leq
        \left( \prod_{i=1}^{k} \frac{2}{(2n-2t_i)^2} \right) \,
        \exp \mathopen{}\left( - |U| \log \frac{n}{n-T} - 2\sum_{i=1}^{k}\log \frac{n}{n-t_i+1} \right)\mathclose{} (1+\lo(1)) \\
        =
        \left( \frac{1}{2n^2} \right)^k \left( 1 - \frac{T}{n} \right)^{\card{U}} (1 + \lo(1)),
    \end{multline*}
    as we claimed.
\end{proof}

\begin{rem}
    Our multiplicative constant $c$ becomes worse starting from this lemma, essentially because a path with hyperbolic length $c \log(n)$ can have combinatorial length of the order $n^{c/2}$.
    Consequently, we have to allow $k$ in the previous lemma to vary from $0$ to $\bO(n^{c/2})$, rather than $\log(n)$ as in \cite{LinialSimkin}.
    As a result, roughly speaking, we have to kill the factor $n^{c/2}$ with $n^{-a}$, and so we need the additional assumptions $a > c/2$.
    While in \cite{LinialSimkin}, the factor $\log n$ can be killed by $n^{-a}$ with $a$ being arbitrarily small.
\end{rem}

The following corollary of the preceding lemma, is the analogue of \cite[Lemma~2.7]{LinialSimkin}.
\begin{lem}\label{lem_probability_bound_k_chords}
    Let $S$ be a set of $k \leq n^{c/2}$ chords, $h_1, h_2 \in \calH^{(0)}$ be distinct half-edges, and let $E$ be the event where every chord in $S$ is chosen by the process, and $h_1, h_2$ are not saturated at time $T+1$.
    Then we have
    \[
        \PP(E)
        =
        \bO \mathopen{}\left( \frac{1}{(2n)^{k+2(1-c-a)}} \right)\mathclose{}.
    \]
\end{lem}
\begin{proof}
    It follows from Lemma~\ref{lem:2.6} that
    \begin{multline*}
        \PP(E)
        \leq
        \sum_{\substack{I \subset\{0, \dots, T \} \\ \card{I} = k}} 
        \left( \frac{1}{2n^2} \right)^k \left( 1 - \frac{T}{n} \right)^2 (1 + \lo(1)) \\
        =
        \bO (T^k) \left( \frac{1}{2n^2} \right)^k \left( 1 - \frac{n-n^{a+c}}{n} \right)^2
        =
        \frac{\bO(n^k)}{(2n^2)^k} \, (n^{a-c-1})^2,
    \end{multline*}
    as claimed.
\end{proof}

In order to be able to speak of paths that might turn into short geodesics in the process we set the following definition:

\begin{dff}
Let $r \geq 3$ be an integer.
A path $P$ in $K_{2n}$ is said to be $r$-\emph{threatening} if
\begin{itemize}
\item the first and the last edges of $P$ are not chords;
\item no two consecutive edges in $P$ are chords;
\item if $P$ carries the word $L^s$ or $R^s$ for some $s$, then $s^2+2 \leq r$, otherwise $\tr(w(P)) \leq r$.
\end{itemize} 
\end{dff}

\begin{lem}\label{lem_number_threatening_paths}
    For any integer $r \geq 3$,
    the number of $r$-threatening paths with $k$ chords in $K_{2n}$ is $\bO(r^2 \log(r) \, (2n)^{k+1})$.
\end{lem}

\begin{proof}
    To count these paths, we first argue in terms of trace and then we sum over the number of possible traces. 
    
    A threatening path is uniquely determined by its starting half-edge and its word. We first bound the number words that can arise. Either it is a word whose trace lies between $3$ and $r$ or it is a word of the form $L^s$ or $R^s$ for some $s$, determined by the trace length.
    For the former there are $\bO(r^2 \log r)$ choices by Proposition~\ref{prp_counting}.
    For the latter, there are $\bO(\sqrt{r})$ choices (and $2$ per $s$).
    In total this makes for $\bO(r^2 \log r)$ possibilities.
    
    Once the word is given, we need a starting half-edge, for which there are $2n$ choices.
    The combination of the word and this half-edge determines the first half-edge of the first chord.
    So, for this first chord, we need to choose a second half-edge, which gives us $\leq 2n$ choices, after which the first half-edge of the second chord is determined by the word again. Repeating this process leads to an upper bound of $(2n)^{k+1}$ for the choices of all the half-edges involved. After these choices, the path is determined.
\end{proof}

Given a graph $\calG$ and an integer $r \geq 3$, we denote by $P_r(\calG)$ the number of $r$-threatening paths in $\calG$.
\begin{lem}\label{lem_threatening_path_bounds}
    For any $3 \leq r \leq \tau_0$, we have
    \[
        P_r(\calG^{(T+1)})
        \leq
        (2n)^{-1+2c+2a} \, r^3 \log r \cdot \log n
    \]
    with high probability.
\end{lem}
\begin{proof}
    Let $P_{r,k}(\calG)$ denote the number of such paths that have $k$ chords.
   Using Lemma~\ref{lem_tr_lower_bound}, we get that the number of chords in a word of trace between $3$ and $r$ is at most $r-1$, and so
   \[
       \EE(P_r(\calG^{(T+1)}))
       =
       \sum_{k=0}^{r-1} \EE(P_{r,k}(\calG^{(T+1)})).
   \]
   Thus, by Lemma~\ref{lem_probability_bound_k_chords} and \ref{lem_number_threatening_paths}, we obtain
   \[
       \EE(P_r(\calG^{(T+1)}))
       \leq
       \sum_{k=0}^{r-1}
       \bO((2n)^{k+1} \, r^2 \log r) \cdot \bO((2n)^{-k-2(1-c-a)})
       =
       \bO((2n)^{-1+2c+2a} \, r^3 \log r).
   \]
   Thus, Markov's inequality implies
   \[
        \PP \left( P_r(\calG^{(T+1)}) \geq (2n)^{-1+2c+2a} \, r^3 \log r \cdot \log n \right)
        =
        \bO(1/\log n).
    \]
    This completes the proof.
\end{proof}

\begin{rem}
    Again, here the bound of $P_r(\calG^{T+1})$ is weaker than that in \cite{LinialSimkin} because the number of chords needs to vary from $0$ to $n^{c/2}$ instead of just up to $\log n$.
    This is another reason why our constant $c$ is worse that in \cite{LinialSimkin}.
\end{rem}

Now we are ready to prove the main result of the section.
\begin{proof}[Proof of Theorem~\ref{thm_linsim_rephrased}]
    By Lemma~\ref{lem_threatening_path_bounds}, we have, with high probability,
    \[
        P_{\tau_0}(\calG^{(T+1)})
        \leq
        \bO(\tau_0^3 \log \tau_0 \log n / n^{1-2c-2a})
        =
        \bO(n^{7c/2 + 2a - 1} \log^2 n).
    \]
    Under the assumptions \eqref{eq:ca} for $c$ and $a$, we can take $c < 2/9$ (and $a \in (c/2, 1/9)$), then we have
    \[
        \frac{7c}{2} + 2a - 1 < 0.
    \]
    As a result,
    \[
        P_{\tau_0}(\calG^{(T+1)})
        \leq
        \bO(n^{7c/2 + 2a - 1} \log^2 n)
        =
        \lo(1).
    \]
    Thus, $\calG^{(T+1)}$ is safe.
    Now the result follows from Lemma~\ref{lem:safe}.
\end{proof}

\begin{rem}
    We failed to make the ``nibbling'' part of \cite{LinialSimkin} work in the hyperbolic context.
    More precisely, the first significant obstacle occurs during our attempt to adapt the proof of \cite[Claim~2.12(c)]{LinialSimkin}.
    Roughly speaking, this is due to the following fact.
    In a graph, the total length of two paths of length $\ell_1$ and $\ell_2$ respectively is simply $\ell_1 + \ell_2$.
    In contrast to this, while for typical words $w_1$ and $w_2$ (in $L$ and $R$) we do expect that $\tr(w_1 w_2)$ should be close to $\tr(w_1) \cdot \tr(w_2)$, but in general, the best bound we have is $\tr(w_1) \leq \tr(w_1 w_2)$, which is far from being enough for our purposes.
\end{rem}

\subsection{The maximal systole problem}\label{sec_maxsys}

In this section we prove Theorem~\ref{thm_liminf}:
\begin{thmrep}{\ref*{thm_liminf}}
We have
\[
\liminf_{g\to\infty} \frac{\max\st{\sys(X)}{X\in\calM_g}}{\log(g)} \geq \frac{2}{9}.
\]
\end{thmrep}

The reason that Theorem~\ref{thm_main_linsim} does not directly imply it, is that we don't control the number of cusps of the resulting surface well enough. In particular, we can't immediately guarantee that we produce a surface of every genus. 

In order to remedy this issue, we modify the procedure in two ways:
\begin{enumerate}
\item We make it less random: as soon as $\calG^{(t)}$ is safe (the definition of which remains the same), we stop the random process and glue the remaining sides together, using any of the (typically many) combinatorial patterns that minimizes the resulting number of cusps.
\item During the random part of the process, we forbid creating
\begin{itemize}
\item cusps in the interior of the surface, and
\item boundary components containing only one side of a triangle
\end{itemize}
That is, we remove all pairs of half-edges $\{h_1,h_2\}$ that correspond to  sides on the boundary of the surface corresponding to $\calG^{(t)}$ that are either
\begin{itemize}
\item adjacent, or
\item have only one other side between them,
\end{itemize}
from the set of available edges $\calA^{(t)}$. 
\end{enumerate}

We first prove a topological lemma. Recall that we build our surface out of $2n$ ideal triangles.
\begin{lem}
Suppose $n$ is odd. If the modified procedure saturates (i.e.~yields a surface without boundary) then the resulting surface has genus $(n+1)/2$ and one cusp.
\end{lem}

\begin{proof}
We first prove that, if this new version of the process saturates, it produces a surface with at most two cusps.

Indeed, at the first time $t$ during which $\calG^{(t)}$ is safe, all the cusps of the corresponding surface are still on the boundary, due to the second modification above. In particular, at this point, the surface still has boundary. We don't know how many boundary components the surface has.  Using the fact that the surface does not have any boundary components containing only one side, we can however glue all components into a single boundary component without creating cusps in the interior of the surface.

So now we have a surface that has one boundary component with an even number, say $2s$, of sides on it.
We can glue these sides together two by two in such a way that the resulting surface has one or two cusps, depending on whether $2s$ is $0$ or $2$ modulo $4$, respectively:
if $s$ is even, then the standard presentation of the fundamental group of a closed genus $s/2$ surface as a $2s$-gon gives the desired gluing (all vertices are identified).
If $s$ is odd, we pick a pair of adjacent sides and glue them together.
Now the boundary component becomes a $(2s-2)$-gon, and we perform the same gluing as above to obtain a surface with two cusps.
By minimality (the first modification), the resulting surface will hence have at most two cusps.

In order to conclude that the surface has one cusp, we compute the Euler characteristic of its compactification, using the given triangulation and writing $C$ for the number of cusps:
\[
V-E+F = C -3n + 2n = C-n = 2-2g.
\]
Indeed, it follows from the last equality that the parity of $C$ and $n$ needs to be the same.
\end{proof}

Now we are ready to prove Theorem~\ref{thm_liminf}:
\begin{proof}[Proof of Theorem~\ref{thm_liminf}]
The lemma above says that, if the process saturates with high probability for all large enough odd $n$, we produce a surface of genus $g$ and one cusp, with systole $\geq c \cdot \log(g)$ for all large enough $g$. Using \cite[Lemma 3.1]{Brooks}, we can compactify these surfaces so as to obtain a sequence of closed surfaces of genus $g$ with systole $\geq (c + \lo(1))\cdot \log(g)$ for all large enough $g$. This proves that
\[
\liminf_{g\to\infty} \frac{\max\st{\sys(X)}{X\in\calM_g}}{\log(g)} \geq c.
\]

In order to prove that the process does saturate, we go through the same proof as in the previous section. We will just describe here what needs to be modified.

We again need to show, that with high probability, $\calG^{(T)}$ is safe. The only real change is in Lemma~\ref{lem:2.3}. The statement remains true, but we have modified the definition of available pairs of half-edges, so the proof needs to be adapted. Compared to the original process, the half-edges we now forbid, are the pairs of half-edges that define sides on the same boundary component at time $t$, such that this pair of sides has either no other sides or just one other side between them. The number such pairs is at most $4 \cdot |\calH^{(t)}|$. This in particular means that the lower bounds in items ii and iii of Lemma~\ref{lem:2.3} remain valid.

After this, we don't change the definition of safety, so Lemma~\ref{lem:safe} remains valid. The proof of Lemma~\ref{lem:2.6} uses Lemma~\ref{lem:2.3}, but not the definition of availability, so it goes through verbatim. Lemma~\ref{lem_probability_bound_k_chords}, only uses Lemma~\ref{lem:2.6}, which thus also remains in tact. We also don't change the definition of threatening paths, so  Lemmas~\ref{lem_number_threatening_paths} and \ref{lem_threatening_path_bounds}, as well as the proof of saturation (as long as $c<\frac{2}{9}$) go through as written.
\end{proof}

\subsection{Counting}\label{sec_counting_belyi}

Before we pass on to regular covers, we prove Theorem~\ref{thm:many}~(a), that we repeat here for convenience:

\begin{thmrep}{\ref*{thm:many}~(a)}
Fix any constant $c < \frac{2}{9}$. Then there are at least $ 
n^{n+\lo(n)}$
pairwise non-isometric cusped hyperbolic surfaces $X_n$ of area equal to $2\pi n$ and systole at least $c \log(\area(X_n))$ that can be obtained by gluing together $2n$ ideal hyperbolic triangles without shear.
\end{thmrep}

\begin{proof}
Our proof will use a combination of ideas and results due to Linial--Simkin \cite[Theorem~1.2]{LinialSimkin}, Belolipetsky--Gelander--Lubotzky--Shalev \cite{BGLS} and Lubotzky \cite{Lubotzky}. In order to use these ideas, we will first explain how the surfaces we produce relate to subgroups of $\Gamma=\PSL(2,\ZZ)$.

Indeed, all the surfaces we produce are (typically non-regular) covers of $\Gamma \backslash \HH^2$. To make this explicit, let $X$ be a connected, oriented and complete hyperbolic surface of finite area that is triangulated with $2n$ ideal triangles with gluings of shear $0$. Label the sides of these triangles with the numbers $1, \ldots, 6n$ (in such a way that each label appears exactly once). This allows us to write down two permutations $\sigma, \tau \in \sym_{6n}$, defined as follows:
\begin{itemize}
\item $\sigma$ is a product of $2n$ distinct $3$-cycles, each cycle containing the labels corresponding to one triangle, in the order that corresponds to the orientation of $X$.
\item $\tau$ is a product of $3n$ distinct $2$-cycles, each cycle containing the labels adjacent to a single edge (coming from the two triangles that are adjacent to it).
\end{itemize}
Observe that $\sigma$ and $\tau$ are of order three and two respectively. This means that there exists a unique homomorphism
\[
\varphi \colon \Gamma \simeq (\ZZ/2\ZZ) \ast (\ZZ/3\ZZ) \longrightarrow \sym_{6n}
\]
that maps the generator of order two of $\Gamma$ to $\tau$ and the generator of order three to $\sigma$. The fact that $X$ is connected implies that the subgroup $\langle \sigma,\tau\rangle < \sym_{6n}$ acts transitively on $\{1,\ldots,6n\}$ and hence that
\[
H \coloneqq \stab_\varphi(\{1\}) < \Gamma
\]
is a subgroup of index $6n$, thus yielding an orbifold cover of degree $6n$ of $\Gamma\backslash \HH^2$. We claim that this orbifold cover can be identified with $X$. Namely, an ideal triangle can be decomposed into three pairwise isometric fundamental domains for $\Gamma$. So we can make the labeling of $X$ correspond to copies of fundamental domains of $\Gamma$. $\stab_\varphi(\{1\})$ can then be identified with the fundamental group of $X$ based at fundamental domain labeled $1$. The fact that the shears of the triangulation are $0$ implies that the induced covering map $X\to\Gamma\backslash\HH^2$ is (homotopic to) the hyperbolic orbifold cover we described above.

The conclusion of the above is that in order to count the number of surfaces we produce with the construction from Theorem~\ref{thm_linsim_rephrased}, we can count how many subgroup of index $6n$ of $\Gamma$ we produce and afterwards count how many of these yield non-isometric surfaces. 

To count subgroups, we first count the number of pairs $(\sigma, \tau)$ we produce. We will assume $X_n^{(0)}$ is an annulus, whose dual graph contains a single cycle carrying the word $(LR)^n$ (like in Figure~\ref{pic_init}). We have $(6n)!/(2\cdot n)$ inequivalent ways to label the sides of $X_n^{(0)}$. Every such labeling determines $\sigma$ and part of $\tau$. The remaining part of $\tau$ is determined by the iterative part of the process. 

Following Linial--Simkin, we first count the number of different runs of the algorithm (which might have the same outcome in terms of $\sigma$ and $\tau$). By Lemma~\ref{lem:2.3}.iii, at each time $0 \leq t \leq T$, the process picks one available edge among at least $2(n-t)^2(1 + \lo(1))$.
    As such, the process can evolve in at least
    \[
        \prod_{t=0}^T 2(n-t)^2 (1 + \lo(1))
        =
        \left( 2n^2(1 + \lo(1)) \right)^{T+1}
        \prod_{t=0}^{T} \left( 1 - \frac{t}{n}\right)^2 
         = n^{2n + \lo(n)}
    \]
    different ways. The probability that the process is successful is $(1+\lo(1))$, implying that the number of successful runs is at least
\[
          n^{2n + \lo(n)}
         \]
The same (labeled) outcome can result from different initial configurations and different runs.
    There are at most $3$ Hamiltonian cycles that carry the word $(LR)^n$ in a trivalent graph with $2n$ vertices.
    During the process, $n$ edges are created, and this can be done in $n!$ different orders.
    Therefore, the process can produce at least
    \[
        \frac{1}{n! \cdot 3} \frac{(6n)!}{2\cdot n} \, n^{2n + \lo(n)} = n^{7n + \lo(n)}
    \]
    labeled triangulated surfaces (pairs $(\sigma,\tau)$) with logarithmic systoles. 
    
Now we use the fact that the map 
\[
\st{\varphi\in\Hom(\Gamma,\sym_{6n})}{\begin{array}{c}\varphi(\Gamma) \text{ acts transitively} \\ \text{on }\{1,\ldots,6n\} \end{array}} \longrightarrow \st{H<\Gamma}{[\Gamma:H] = 6n}
\]
given by $\varphi\mapsto \stab_\varphi(\{1\})$ is $((6n-1)!)$-to-$1$ (see for instance \cite[Proposition~1.1.1]{LubotzkySegal}). This means that we obtain at least 
\[
    \frac{1}{(6n-1)!} \, n^{7n + \lo(n)}
    =
    n^{n + \lo(n)}
\]
subgroups of index $6n$ in $\Gamma$.

In order to determine how many of these subgroups yield isometric surfaces, we need to count how many of them are conjugate in $G=\PSL(2,\RR)$. If $H_1,H_2<\Gamma$ are two subgroups of index $6n$ and $h\in G$ is such that $h H_1 h^{-1} = H_2$, then $h\in \mathrm{Comm}_G(\Gamma)$, the commensurator of $\Gamma$ in $G$. The problem is that $\Gamma$ (being an arithmetic group) is of infinite index in its commensurator, so this leaves lots of possibilities for $h$.

To solve this issue, we apply an argument due to Belolipetsky--Gelander--Lubotzky--Shalev \cite[p.\ 2218]{BGLS}. The \emph{congruence closure} $\overline{H}^{\mathrm{cong.}}$ of a subgroup $H<\Gamma$ is the smallest congruence subgroup of $\Gamma$ containing $H$. Lubotzky \cite{Lubotzky} proved that $\Gamma$ has less than $n^{\lo(n)}$ congruence subgroups of index $\leq 6n$ (Lubotzky's result is a lot sharper than this, but this is all we need). This implies that our procedure produces at least $n^{n + \lo(n)}$ distinct subgroups of index $6n$ with the same congruence closure $H_0$. Belolipetsky--Gelander--Lubotzky--Shalev prove that if 
\[
H_0 = \overline{H_1}^{\mathrm{cong.}} = \overline{H_2}^{\mathrm{cong.}} \quad \text{and} \quad h H_1 h^{-1} = H_2 
\]
then $h \in N(H_0)$, the normalizer of $H_0$ in $G$. The group $N(H_0)$ is a lattice in $G$, which means that its covolume is uniformly bounded from below (by $\pi/42$). This means that 
\[
[N(H_0):H_1] = [N(H_0):H_2] \leq 84 n
\]
and hence that the number of distinct conjugates of a given subgroup $H<\Gamma$ of index $6n$ and with conjugacy closure $H_0$ is at most $84n$. Which means that we obtain at least
\[
\frac{1}{84n} \, n^{n+\lo(n)} = n^{n+\lo(n)}
\]
pairwise non-isometric surfaces.
\end{proof}

\subsection{A database of surfaces with large systoles}\label{sec_examples}
In order to obtain Figure~\ref{fig:examples}, we ran a version of the process above. We modified the conditions of the random algorithm by requiring the gluings not to close up any paths carrying a word of the form $L^k$ up until the very last step of the process.

Recall that in order to make the systole large, we require the process not to call up any loops of length below $2\cosh^{-1}(\tau_0/2)$ where $\tau_0$ is a parameter we fix.
To produce our surfaces, we start with a large $\tau_0$ (around $g$), and iterate the process up to a maximum of $75$ attempts to get one surface with systole at least $\tau$.
If the process fails all these attempts (which is typical at this stage),
then we decrease $\tau_0$ by $1$ and restart again, repeating this cycle until a successful outcome is attained.
Every point in the plot in green represents the Bely\u{\i} surfaces with largest systole for that genus that we found.

We have produced a data base of the surfaces we found, that is attached as an ancillary file (\texttt{random\char`_surfaces\char`_DATABASE.pkl}) to the arXiv version of this paper.
The data is stored within a \texttt{Python} dictionary whose keys correspond to genera, and the content for each genus $g$ is pair:
\[
\Big(\tau_0,\;[i_0, i_1, \dots, i_{12g-7}]\Big).
\]
The first item of the pair is an integer $\tau_0$, and the second item is a list $[i_0, i_1, \dots, i_{12g-7}]$ of distinct integers ranging from $0$ to $12g-7$, representing a Bely\u{\i} surface of $g$ with $1$ cusp and systole at least $2\cosh^{-1}(\tau_0/2)$.
The surface, consisting of $4g-2$ triangles, is encoded as follows.
We label the edges of the these $4g-2$ oriented triangles with integers $0, \dots, 12g-7$ such that each triangle's edges are labeled in counter-clock order (with respect to the orientation) with three consecutive numbers.
Then the surface represented by $[i_0, i_1, \dots, i_{12g-7}]$ is obtained by gluing the edge $k$ with the edge $i_k$ (with a zero shear orientation-reversing gluing). The \texttt{jupyter} notebook \texttt{example.ipynb} shows how to access the data using \texttt{SageMath}.

\section{Random regular covers}\label{sec_covers}

In this section, we investigate constructions coming from random regular covers of a fixed hyperbolic surface. These regular covers are a hyperbolic analogue of the random Cayley graphs studied by Gamburd--Hoory--Shahshahani--Shalev--Vir\'ag \cite{GHSSV}, Eberhard \cite{Eberhard} and Liebeck--Shalev \cite{LiebeckShalev3}. 

Let us first define what we mean with a random regular cover. If $X$ is a hyperbolic surface of finite area, $\Gamma=\pi_1(X)$, $G$ a finite group, we will write $X_G$ for a random $G$-cover of $X$. That is, we pick an element $\varphi \in \Hom(\Gamma,G)$, uniformly at random and then $X_G$ is the cover corresponding to the subgroup $\ker(\varphi) \triangleleft \Gamma$. Technically $X_G$ is only a $G$-cover when $\varphi$ is surjective. However, in most of the cases we will study, it will either be surjective with high probability, or we will be able to control its image.

\subsection{Non-regular covers}

Before we get to the proof for regular covers, let us briefly comment on more general covers. In fact, the reason we consider regular covers is that it follows from the results of Nica \cite{Nica} and Magee--Naud--Puder \cite{MageePuder,MageeNaudPuder} that with probability tending to $1$ as $n\to\infty$, the systole of a random $n$-sheeted cover of a hyperbolic surface of finite area is bounded. So a priori, random finite covers might not seem like a good place to look for surfaces of large systole.

However, work by Dixon \cite{Dixon} and Liebeck--Shalev \cite{LiebeckShalev1} (see also \cite{MuellerSchlagePuchta1, MuellerSchlagePuchta2}) implies that with high probability as $n\to\infty$, a random cover of degree $n$ of a fixed hyperbolic surface of finite area is not regular.
Indeed, this can be seen from the fact that $n$-sheeted covers can be identified with transitive permutation actions on the set $\{1,\ldots,n\}$.
If the cover were regular, then the subgroup of the symmetric group $\sym_n$ generated by the permutation action would have $n$ elements. However, with probability tending to $1$ as $n\to\infty$, it contains the alternating group $\alt_n$. This in particular means that the results on random covers do not automatically destroy all hopes for regular covers. In fact, it will turn out that, random regular covers, just like random Cayley graphs, are a great source for surfaces of large systole.

\subsection{The union bound}

In order to control the systole of a random $G$-cover, we use  a similar union bound to that of the papers on random Cayley graphs. That is, we have
\begin{equation}\label{eq_unionbound}
        \begin{split}
            \mathbb{P} \big( \mathrm{sys}(X_G) < R \mkern2mu \big)
    & =
    \mathbb{P} \big( \exists \, \gamma^\Gamma \text{ such that } \gamma \in \ker(\varphi),\ \tau(\gamma) < R \big) \\
    & \leq
    \sum_{\gamma^\Gamma : \, \tau(\gamma) < R} \mathbb{P} \big( \gamma \in \ker(\varphi) \big)    
        \end{split}
\end{equation}

where $\gamma^\Gamma$ and $\tau(\gamma)$ denote conjugacy class of $\gamma$ in $\Gamma$ and the translation length of $\gamma$ on $\HH^2$ respectively. The number of terms in this bound is asymptotic to $e^R/R$ as $R\to\infty$ \cite{Huber_primegeod,Sarnak_Thesis,Venkov}. So, as in the case of graphs, we need to understand how the probability $\PP(\gamma \in \ker(\varphi))$ depends on the translation length of $\gamma$. 

The difference with respect to graphs is that for us this is the hyperbolic translation length instead of the word length of $\gamma$ with respect to some free generating set. On top of that, if $X$ is closed, its fundamental group is not free, which makes the structure of $\Hom(\Gamma,G)$ more complicated. For most of this section, $G$ will be a finite group of Lie type, we will comment on the case of symmetric groups in the end.

\subsection{Closed base surfaces and groups of Lie type}

We will treat closed and non-compact base surfaces of finite area seperately because the issues we encounter depend heavily on which of the two cases we are working with. We will start with closed surfaces. We will prove:

\begin{thm}\label{thm_randcoverscompact}
Let $X$ be a closed orientable hyperbolic surface. Moreover, let $\mathbf{G}$ be a connected, simply connected, simple algebraic group defined over $\QQ$. Then there exists a prime $p_0$ such that, with high probability as $\card{\FF}\to\infty$,
\[
\sys(X_{\mathbf{G}(\FF)}) \geq \bigg( \frac{1}{\dim(\mathbf{G})} + \lo(1) \bigg) \cdot \log(\area(X_{\mathbf{G}(\FF)})),
\]
for any sequence of finite fields $\FF$ whose characteristics all exceed $p_0$. 
\end{thm}

Before we get to the proof of this theorem, we record two lemmas that we will need. The first of these is a standard consequence of the \u{S}varc--Milnor lemma (see for instance \cite[Proposition 8.19]{BridsonHaefliger}).

\begin{lem}\label{lem_svarcmilnor_translength}
Let $S$ be a finite generating set of $\Gamma$ and let $\tau_S(\gamma)$ denote the translation length of $\gamma \in \Gamma$ on the Cayley graph $\mathrm{Cay}(\Gamma, S)$ of $\Gamma$ with respect to $S$.
Then there exists a constant $K_{X,S}>0$ such that
\[
\frac{1}{K_{X,S}} \cdot \tau_S(\gamma) - K_{X,S} \leq \tau(\gamma) \leq K_{X,S} \cdot  \tau_S(\gamma) \quad \text{for all }\gamma\in\Gamma.
\]
\end{lem}

\begin{proof}
Since we could not find a proof in the literature, we provide one for completeness.  We recall that if a group $\Gamma$ acts on a metric space $(X,d)$ by isometries, then the \emph{stable translation length} of an element $\gamma\in \Gamma$ is given by:
\[
\overline{\tau}_X(\gamma) \coloneqq \lim_{n\to\infty} \frac{d(\gamma^n x_0,x_0)}{n}
\]
for any $x_0\in X$. It follows from the triangle inequality that this definition does not depend on the base point $x_0$.

So, fix a base point $x_0\in\HH^2$ to define an orbit map $\Gamma\to\HH^2$ by $\gamma\mapsto \gamma\cdot x_0$. By the \u{S}varc--Milnor lemma,
\[
\frac{1}{K_{X,S}'}\cdot d_{\mathrm{Cay}(\Gamma,S)}(\gamma^n,e) - K_{X,S}' \leq d_{\HH^2}(\gamma^n x_0,x_0) \leq 
K_{X,S}'\cdot d_{\mathrm{Cay}(\Gamma,S)}(\gamma^n,e) + K_{X,S}'
\]
for some constant $K_{X,S}'>0$ independent of $\gamma$ and $n$. Dividing by $n$, taking limits and using the fact that on the hyperbolic plane, stable translation length and translation length coincide, we obtain
\[
\frac{1}{K_{X,S}'}\cdot \overline{\tau}_S(\gamma)  \leq \tau(\gamma) \leq 
K_{X,S}'\cdot \overline{\tau}_S(\gamma).
\]
Finally, we use that $\Gamma$ is a hyperbolic group, which implies that there exists a uniform constant $K''_S$ such that
\[
\tau_S(\gamma) - K''_S \leq \overline{\tau}_S(\gamma) \leq \tau_S(\gamma)
\]
for all $\gamma\in\Gamma$. Stringing the last two inequalities together proves the lemma.
\end{proof}

The second lemma is about evaluation maps: given $\gamma \in \Gamma$, the associated evaluation map $\mathrm{ev}_\gamma:\Hom(\Gamma,G) \to G$ is given by
\[
\mathrm{ev}_\gamma(\varphi) = \varphi(\gamma), \quad \text{for all } \varphi\in\Hom(\Gamma,G).
\]
This is the analog of a word map in our setting. Our goal is to understand the sets
\[
V_{G,\gamma} = \st{\varphi \in \Hom(\Gamma,G)}{\mathrm{ev}_\gamma(\varphi)=e}.
\] 
If $G$ is an algebraic group, then $\Hom(\Gamma,G)$ is a algebraic variety that can be identified with
\[
\st{(A_1,B_1,\ldots,A_g,B_g) \in G^{2g}}{\prod_{i=1}^g [A_i,B_i] = e}
\]
and $V_{G,\gamma}$ is a subvariety.
\begin{lem}\label{lem_dimbound}
If $\mathbf{G}$ is a connected, simply connected, simple algebraic group defined over $\QQ$. Then there exists a prime $p_0$ such that for all algebraically closed fields $\FF$ of characteristic $p>p_0$ and for all $\gamma\in\Gamma$,
\[
\dim(V_{\mathbf{G}(\FF),\gamma}) < \dim(\Hom(\Gamma,\mathbf{G}(\FF))).
\]
whenever $\mathbf{G}$ is defined over $\FF$.
\end{lem}

\begin{proof} For ease of notation, we'll write $G \coloneqq \mathbf{G}(\FF)$. Weil \cite{Weil} proved that $\Hom(\Gamma,\mathbf{G}(\CC))$ is non-singular,
and Li \cite[Theorem~0.1]{Li} proved that it is connected,
which means that it is irreducible (see also Liebeck--Shalev \cite[Corollary~1.11(ii)]{LiebeckShalev2} for related results). 
By the Bertini--Noether theorem (see for instance \cite[Proposition~10.4.2]{FriedJarden}) this means that $\Hom(\Gamma, G)$ is irreducible whenever the characteristic of $\FF$ is large enough. 

So all we need to do, is prove that $V_{G,\gamma}$ is a proper subvariety. To this end, we use classical results by Baumslag and Borel: $\Gamma$ is residually free \cite{Baumslag} and hence there is an epimorphism $\psi \colon \Gamma \to F_r$, where $F_r$ is a free group of rank $r \geq 1$, such that $\psi(\gamma)\neq e$. By Borel's theorem on the dominance of word maps \cite[Theorem B]{Borel}, we can find a homomorphism $\rho \colon F_r \to G$ such that $\rho(\psi(\gamma)) \neq e$. In particular, $\rho \circ \psi \in \Hom(\Gamma,G) \smallsetminus V_{G,\gamma}$. 
\end{proof}

\begin{rem}\label{rem_irreducibility_problem}
It would be nice to have a version of the lemma above that doesn't use the assumption that $\mathbf{G}$ is simply connected. However, the proof based on the results by Borel and Baumslag does not work in this case. Indeed, when $\mathbf{G}$ is not simply connected, the irreducible components of $\Hom(\Gamma,\mathbf{G}(\CC))$ are indexed by the elements of the fundamental group $\pi_1(\mathbf{G}(\CC))$ (\cite[Theorem 0.1]{Li}, \cite[Corollary 1.11(ii)]{LiebeckShalev2}). The representations that factor through a map to a free group all lift to representations $\Gamma \to \mathbf{\widetilde{G}}(\FF)$, where $\mathbf{\widetilde{G}}$ denotes the universal cover of $\mathbf{G}$. This means that they all sit in an image of $\Hom(\Gamma,\mathbf{\widetilde{G}}(\FF))$, which is irreducible, and hence are all elements of the same component of $\Hom(\Gamma,\mathbf{G}(\FF))$.
\end{rem}

We are now ready to prove Theorem~\ref{thm_randcoverscompact}:

\begin{proof}[Proof of Theorem~\ref{thm_randcoverscompact}]
For ease of notation, we'll write $G=\mathbf{G}(\FF)$ again. Recall from \eqref{eq_unionbound} that
\begin{equation}\label{eq_unionboundfilledin}
    \mathbb{P} \big( \mathrm{sys}(X_G) < R \mkern2mu \big)
    \leq  \sum_{\gamma^\Gamma : \, \tau(\gamma) < R} \mathbb{P} \big( \gamma \in \ker(\varphi) \big)   \leq
    \frac{|\{ \gamma^\Gamma : \tau(\gamma) < R \}|}{|\Hom(\Gamma, G)|}
    \cdot \sup_{\gamma^\Gamma : \, \tau(\gamma) < R} |V_{G, \gamma}|.
\end{equation}

The dimension bound of Lemma~\ref{lem_dimbound} and the assumption that the characteristic of $\FF$ exceeds $p_0$ imply that
\[
    |V_{G, \gamma}|
    \leq
    \deg \mathopen{}\big( V_{\mathbf{G}(\overline{\FF}), \gamma} \big)\mathclose{} \cdot |\FF|^{\dim V_{\mathbf{G}(\overline{\FF}), \gamma}}
    \leq
    \deg \mathopen{}\big( V_{\mathbf{G}(\overline{\FF}), \gamma} \big)\mathclose{} \cdot |\FF|^{\dim \Hom(\Gamma, \mathbf{G}(\overline{\FF})) - 1},
\]
see for instance \cite[Claim~7.2]{ZeevKollarLovett} or \cite[Lemma~A.3]{EllenbergOberlinTao}. Here, $\deg \mathopen{}\big( V_{\mathbf{G}(\overline{\FF}), \gamma} \big)\mathclose{} $ is the degree of $V_{G,\gamma}$ is the algebraic closure of $\FF$.

The degree $\deg \mathopen{}\big( V_{\mathbf{G}(\overline{\FF}), \gamma} \big)\mathclose{}$ can be bounded by the products of the degrees of the polynomials defining $V_{G,\gamma}$ (see for instance \cite[Example~8.4.6]{Fulton}). These degrees are at most the word length of $\gamma$ with respect to $S$.

Since the probability $\PP(\gamma \in \ker(\varphi))$ is a conjugacy invariant function on $\Gamma$, we are free to assume $\gamma$ realizes the minimal word length in its conjugacy class. This minimal word length is exactly the translation length $\tau_S(\gamma)$ of $\gamma$ on the Cayley graph $\mathrm{Cay}(\Gamma,S)$.

This means that the degrees of the polynomial equations defining $V_{G,\gamma}$ can be chosen to be at most $K_S\cdot\tau(\gamma)$ (by Lemma~\ref{lem_svarcmilnor_translength}). Finally, the number of equations depends only on the size of the matrix representation and hence only on $\mathbf{G}$, so we will write $N_{\mathbf{G}}$ for this number. Combining all the bounds, we obtain the bound
\begin{equation}\label{eq_boundbadhomom}
\card{V_{G,\gamma}} \leq \Big(K_S\cdot \tau(\gamma) \Big)^{N_{\mathbf{G}}} \cdot \card{\FF}^{\dim(\Hom(\Gamma,\mathbf{G}(\overline{\FF})))-1}.
\end{equation}
On the other hand (see for instance \cite[Corollary~3.2]{LiebeckShalev2}), 
\[
\card{\Hom(\Gamma,G)} = \card{G}^{2g-1} \cdot \zeta^G(2g-2),
\]
where $g$ is the genus of $X$ and $\zeta^G$ is the $\zeta$-function of $G$, defined by
\[
\zeta^G(s) = \sum_{\chi \in \mathrm{Irr}(G)} \chi(1)^{-s},
\]
where $\mathrm{Irr}(G)$ denotes the set of equivalence classes of irreducible complex characters of $G$. Observe that, because the trivial representation appears in $\mathrm{Irr}(G)$,
\[
\card{\Hom(\Gamma,G)} \geq \card{G}^{2g-1} \sim \card{\FF}^{(2g-1)\cdot\dim(\mathbf{G}(\overline{\FF}))}\quad \text{as } \card{\FF} \to \infty.
\]
Liebeck--Shalev \cite[Corollary 1.3]{LiebeckShalev2} proved that if $G$ is quasisimple, this bound is asymptotically sharp (but we will not need this below). By \cite[Corollary~1.11]{LiebeckShalev2}, we have $(2g-1)\cdot\dim(\mathbf{G}(\overline{\FF}))=\dim(\Hom(\Gamma,\mathbf{G}(\overline{\FF}))$. So,
\[
\card{\Hom(\Gamma,G)} \gtrsim \card{\FF}^{\dim(\Hom(\Gamma,\mathbf{G}(\overline{\FF})))},\quad \text{as } \card{\FF} \to \infty.
\]
So, filling in \eqref{eq_unionboundfilledin} with this bound, the bound from \eqref{eq_boundbadhomom}, and Huber's asymptotic formula for the number of terms, we obtain
\begin{multline*}
    \PP \big( \sys(X_G) < R \big) \leq \frac{c_X \cdot e^R}{R\cdot \card{\Hom(\Gamma,G)}}\sup_{\gamma^\Gamma:\;\tau(\gamma)<R} \Big(K_S\cdot \tau(\gamma) \Big)^{N_{\mathbf{G}}} \cdot \card{\FF}^{\dim(\Hom(\Gamma,\mathbf{G}))-1}  \\
\leq \frac{c_{X,\mathbf{G}} R^{N_{\mathbf{G}}-1} e^R}{\card{\FF}},
\end{multline*}
for some constants $c_X,c_{X, \mathbf{G}} > 0$.
Thus, we obtain that, with high probability as $\card{\FF}\to\infty$ among fields with characteristic $>p_0$, 
\[
\sys(X_{\mathbf{G}(\FF)})\geq (1 + \lo(1)) \cdot \log(\card{\FF}).
\]
Now using that 
\[
\log(\card{\FF}) \sim \frac{1}{\dim(\mathbf{G})}\cdot \log(\card{\mathbf{G}(\FF)}) \sim \frac{1}{\dim(\mathbf{G})}\cdot \log(\area(X_{\mathbf{G}(\FF)})),
\]
as $\card{\FF}\to\infty$, we obtain the bound we claimed.
\end{proof}

\subsection{Non-compact base surfaces and groups of Lie type}

Now we treat the case in which $X$ is non-compact. In this case, $\Gamma \coloneqq \pi_1(X)$ is free and hence we can use the bounds on the probability that a word evaluates to the identity undelying \cite[Theorem~4]{GHSSV}, so we can now treat all finite simple groups of Lie type. On the other hand, in this case the relation between the word length of an element of $\Gamma$ and its translation length on the hyperbolic plane is more complicated. In practice, this means that our multiplicative constants will get worse. We will prove:

\begin{thm}\label{thm_randcovers_noncompact}
Let $X$ be a non-compact hyperbolic surface of finite area. Moreover, let $\mathbf{G}$ be a simple algebraic group, defined over $\FF_q$ for all $q$. Then there exists a constant $a_X>0$ such that
\[
\PP\left(\sys(X_{\mathbf{G}(\FF_q)}) > a_X\cdot \log(\area(X_{\mathbf{G}(\FF_q)})) \right) \to 1
\]
as $q\to\infty$.
\end{thm}

The proof of this theorem still relies on the union bound \eqref{eq_unionbound}. Since we are no longer assuming the surface is closed, $\Gamma$ is no longer quasi-isometric to $\HH^2$ and we can no longer use Lemma~\ref{lem_svarcmilnor_translength}. We replace this by the following:
\begin{lem}\label{lem_distortion}
For any finite generating set $S$ of $\Gamma$ there exists a constant $K_{X,S}>0$ such that
\[
\frac{1}{K_{X,S}}\cdot \log(\tau_S(\gamma)) \leq \tau(\gamma) \leq K_{X,S} \cdot \tau_S(\gamma),
\]
for all hyperbolic elements $\gamma \in \Gamma$, where $\tau_S \colon \Gamma \to \NN$ denotes the translation length of $\gamma$ on the Cayley graph $\mathrm{Cay}(\Gamma,S)$.
\end{lem}

\begin{proof}
The upper bound comes from the proof that is also behind Lemma~\ref{lem_svarcmilnor_translength}, so we will focus on the lower bound.

Because the Cayley graphs $\mathrm{Cay}(\Gamma,S)$ and $\mathrm{Cay}(\Gamma,S')$ are quasi-isometric for any pair of finite generating sets $S$ and $S'$, we can pick our favorite finite generating set. To this end, take a finite sided ideal polygon $P$ in $\HH^2$ that is a fundamental domain for $\Gamma$. Let $S$ be the set of side identifications of $P$ that appear in $\Gamma$. This is a finite and free generating set. Freeness can be seen from the fact that every side of $\gamma(P)$ is separating for all $\gamma\in\Gamma$, which means that all edges of $\mathrm{Cay}(\Gamma,S)$ with respect to $S$ are separating and hence that $\mathrm{Cay}(\Gamma,S)$ is a tree.

Now write $S=\{s_1,s_1^{-1}\ldots,s_k,s_k^{-1}\}$ and let $w=s_{i_1}^{\eps_1}s_{i_2}^{\eps_2}\cdots s_{i_n}^{\eps_n}$, with $\eps_i\in\{\pm 1\}$ for $i=1,\ldots,n$, be a cyclically reduced word conjugate to $\gamma\in\Gamma$, so that $\tau_S(\gamma)=n$. The idea is to study the connected polygon
\[
A = s_{i_1}^{\eps_1}(P)\cup s_{i_1}^{\eps_1} s_{i_2}^{\eps_2}(P) \cup \cdots \cup  s_{i_1}^{\eps_1} s_{i_2}^{\eps_2}\cdots s_{i_n}^{\eps_n}(P),
\]
which consists of $n$ translated copies of $P$ that spell out the word $w$. $w$ maps a side $a$ of $A$ to a side $b$ of $A$. Because hyperbolic fixed points of $\Gamma$ cannot be parabolic fixed points of $\Gamma$ (by discreteness), these sides $a$ and $b$ have distinct endpoints.
Because both $a$ and $b$ separate $\HH^2$, the axis $\alpha_w$ of $w$ runs through both these sides and because $A$ is convex, the part of $\alpha_w$ between $a$ and $b$ is entirely contained in $A$. In particular, we have
\[
\ell(\alpha_w\cap A) = \tau(w) = \tau(\gamma).
\]
So, given that $\tau_S(\gamma) =n $, we need to compare $\ell(\alpha_w\cap A)$ to $n$.

To this end, we set
\[
R \coloneqq \min\st{\dist(a,s_i b)}{\begin{array}{c}i=1,\ldots,k,\;a \text{ and }b \text{ sides of }P \text{ such that } \\ a\text{ and }s_i b \text{ are not asymptotric}\end{array}}>0.
\]
Here we call two distinct geodesics in $\HH^2$ asymptotic if they share an endpoint in $\partial_\infty\HH^2$.

Moreover, we let $\iota(w)$ denote the maximal number of disjoint subsegments of $\alpha_w\cap A$ that connect two non-asymptotic sides of two consecutive copies of $P$. To define this properly, for two integers $\leq l < m \leq k$, we will write
\[
A^{[l,m]} \coloneqq s_ {i_l}^{\eps_l}\cdots s_{i_1}^{\eps_1} (P) \cup  s_ {i_{l+1}}^{\eps_{l+1}}s_ {i_l}^{\eps_l}\cdots s_1^{\eps_1} (P) \cup \cdots \cup s_{i_m}^{\eps_m} \cdots s_ {i_{l+1}}^{\eps_{l+1}}s_ {i_l}^{\eps_l}\cdots s_1^{\eps_1} (P)
\]
for a sub-polygon of $A$ consisting of multiple consecutive copies of $P$, and
\[
\alpha_w^{[l,m]} \coloneqq \alpha_w \cap A^{[l,m]} 
\]
for the intersection of the axis $\alpha_w$ with multiple consecutive copies of $P$ in $A$ and define
\[
\iota(w) \coloneqq \max\st{m<n}{\begin{array}{c} j_1+1 < j_2, \ j_2+1<j_3,\ \ldots , \ j_{m-1}+1<j_m \\
\alpha_w^{[j_1,j_1+1]}, \ \alpha_w^{[j_2,j_2+1]}, \ \ldots, \ \alpha_w^{[j_m,j_m+1]} \text{ all connect two} \\ \text{non-asymptotic geodesics}
\end{array} }.
\]
By construction, we have 
\begin{equation}\label{eq_inversions}
\tau(\gamma)=\ell(\alpha_w\cap A) \geq \iota(w) \cdot R.
\end{equation}
Now suppose that $\tau(\gamma) \leq \log(\tau_S(\gamma))=\log(n)$ (which we may, because if not, then there's nothing to prove). From \eqref{eq_inversions}, we obtain that
\[
\iota(w) \leq \log(n)/R.
\]
This means that $w$ contains a subword $w' = s_{i_l}^{\eps_l}s_{i_{l+1}}^{\eps_{l+1}}\cdots s_{i_m}^{\eps_m}$ with
\[
m-l+1 \geq \frac{n-2\iota(w) }{\iota(w)} \geq
 \frac{R\cdot n}{\log(n)}-2
\]
letters such that the corresponding arc $\alpha_w^{[l,m]}$ connects two asymptotic sides of 
$A^{[l,m]}$. We will assume that $[l,m]$ is maximal with respect to this last property. Because $w$ is hyperbolic, $A^{[l,m]}$ is not all of $A$. So inside of $A$, $A^{[l,m]}$ is adjacent to at least one more copy of $P$. Without loss of generality, we assume this copy is $(m+1)$-st copy. We claim that
\[
\ell(\alpha^{[l,m+1]}) \geq \cosh^{-1} \mathopen{}\left( c_P\cdot (m-l+1) \right)\mathclose{} \geq c_P ' \log(n),
\]
where $c_P,c_P'>0$ are constants that depend on $P$ alone. This can be proved by a direct computation in the hyperbolic plane. Indeed, using the upper half plane model for $\HH^2$, we can uniformize so that the sides (of translates of $P$) that $\alpha_w^{[l,m]}$ intersects, are vertical geodesics in $\HH^2$ and the leftmost side of $s_{i_l}^{\eps_l}\cdots s_{i_1}^{\eps_1}(P)$ is the imaginary axis. Now choose a horosegment $\eta$ in $s_{i_l}^{\eps_l}\cdots s_{i_1}^{\eps_1}(P)$ around the cusp at $\infty$ that does not intersect any of the other sides of $s_{i_l}^{\eps_l}\cdots s_{i_1}^{\eps_1}(P)$. The length $\ell(\alpha_w^{[l,m+1]})$ can be bounded from below by the distance from the imaginary axis to the rightmost translate $\eta'$ of $\eta$ (see Figure~\ref{pic_polygonpath}). This can be computed in terms of $m-l+1$, the Euclidean width of $P$ and the Euclidean height of $\eta$ alone, which leads to the bound above and thus to the lemma.
\begin{figure}
\begin{center}
\begin{overpic}[scale=0.55]{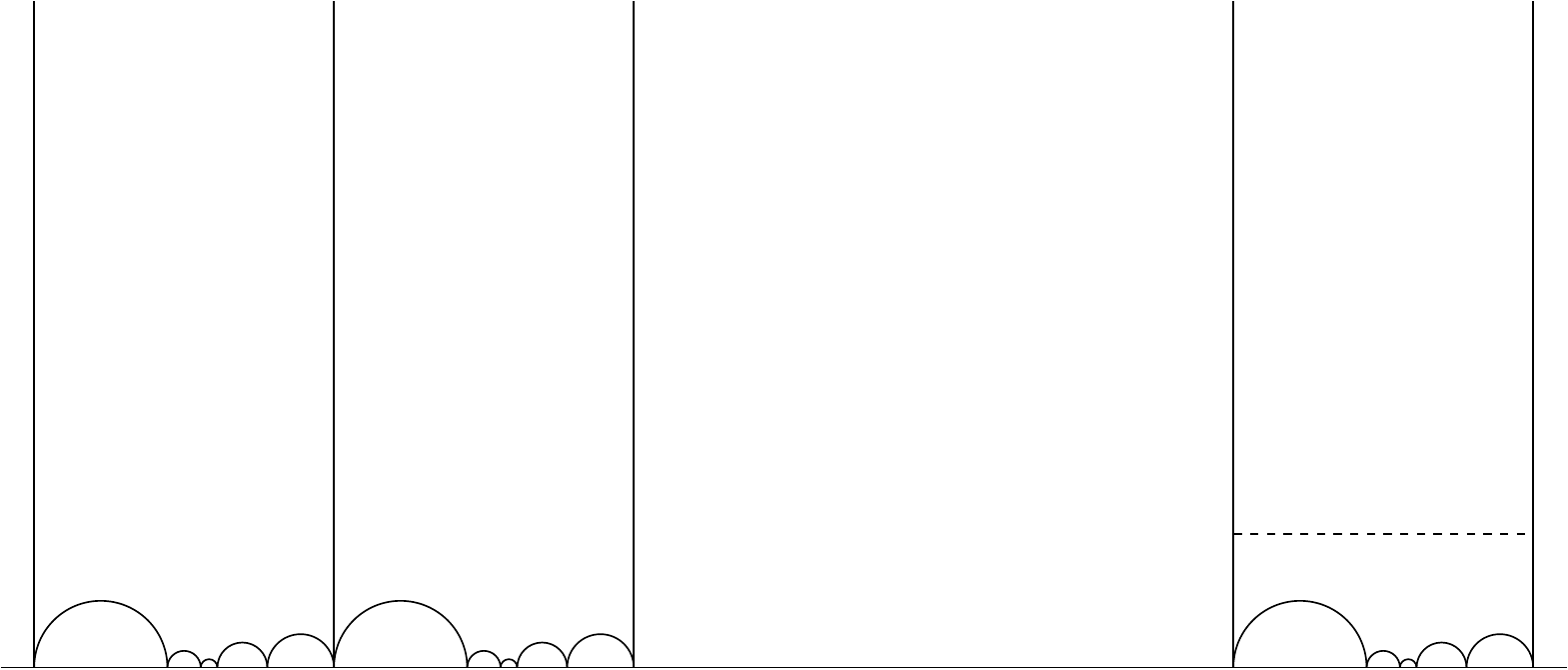}
\put(4.5,20){$s_{i_l}^{\eps_l}\cdots s_{i_1}^{\eps_1}(P)$}
\put(22.4,20){$s_{i_{l+1}}^{\eps_{l+1}}\cdots s_{i_1}^{\eps_1}(P)$}
\put(79.2,20){$s_{i_{m+1}}^{\eps_{m+1}}\cdots s_{i_1}^{\eps_1}(P)$}
\put(48,20){$\cdots$}
\put(58,20){$\cdots$}
\put(68,20){$\cdots$}
\put(81,10){$\eta'$}
\end{overpic}
\caption{$A^{[l,m+1]}$: the union of polygons that the segment $\alpha^{[l,m+1]}$ runs through.}\label{pic_polygonpath}
\end{center}
\end{figure}
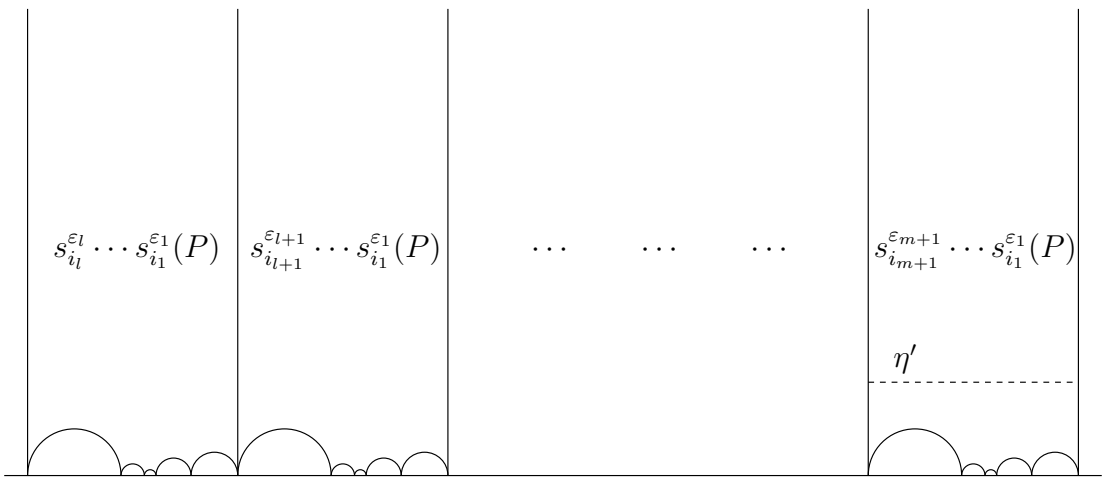
\end{proof}

With the lemma above and the bounds from Gamburd--Hoory--Shahshahani--Shalev--Vir\'ag \cite{GHSSV} we can now prove Theorem~\ref{thm_randcovers_noncompact}:

\begin{proof}[Proof of Theorem~\ref{thm_randcovers_noncompact}]
Let $S$ be a free generating set of $\Gamma$ and let $\varphi\in\Hom(\Gamma,G)$ denote our random homomorphism, by Lemma~\ref{lem_distortion} we get that
\[
\PP\left(\gamma \in \ker(\varphi)\right) \leq \max \st{ 
\PP\left(\beta \in \ker(\varphi)\right)
}
{
\tau_S(\beta) \leq \exp( K\cdot \tau(\gamma))
}.
\]
where $K>0$ depends on $X$ alone. By \cite[Equation~(7)]{GHSSV}, we have
\[
\PP\left(\beta \in \ker(\varphi)\right) \leq \frac{a\cdot \tau_S(\beta)^b}{q},
\]
where $a,b>0$ depend on $\mathbf{G}$ only. Combining this with the union bound \eqref{eq_unionbound} leads to the theorem.
\end{proof}

\subsection{Expansion}

Let us briefly comment on how to derive Corollary~\ref{cor_spectral_gap}:

\begin{correp}{\ref*{cor_spectral_gap}}
Let $X$ be a closed orientable hyperbolic surface, then there exist constants $C,\eps>0$ such that, with high probability as $p\to\infty$,
\[
\lambda_1(X_{\SL_2(\FF_p)}) > \eps \quad \text{and} \quad \mathrm{diameter}(X_{\SL_2(\FF_p)})\leq C\cdot \log(\area(X_{\SL_2(\FF_p)})).\]
Moreover, if $\det(\Delta_p)$ denotes the  $\zeta$-regularized determinant of the Laplacian $\Delta_p$ of $X_{\SL_2(\FF_p)}$, then
\[
\frac{\log(\det(\Delta_p))}{\area(X_{\SL_2(\FF_p)})} \longrightarrow E \quad \text{in probability,}
\]
where $E=(4\zeta'(-1)-1/2+\log(2\pi))/4\pi = 0.0538\ldots$.
\end{correp}

\begin{proof}[Proof sketch] 
A minor modification of our proof also shows that the random Cayley graph $\mathrm{Cay}(\SL(2,\FF),\varphi(S))$, where $S$ is a finite generating set of the fundamental group $\pi_1(X)$, and $\varphi\in\Hom(\pi_1(X),\SL(2,\FF))$ is a uniform random element, has logarithmic girth with high probability. Indeed, we just need to replace hyperbolic length with word length in $\pi_1(X)$ with respect to $S$ in the above.

By the restuls of Bourgain--Gamburd \cite{BourgainGamburd}, this implies that these Cayley graphs have a uniform spectral gap. The Brooks--Burger transfer method \cite{Brooks_transfer,Burger} then yields that the corresponding covers have a uniform spectral gap. The diameter bound in turn follows from Magee's bound \cite{Magee_letter}. Finally, given that, with high probability as $p\to\infty$, the systoles of the surfaces $X_{\SL_2(\FF_p)}$ tend to infinity and their spectral gaps are uniformly bounded away from $0$, they satisfy the hypotheses of \cite[Theorem 3.1]{Naud_determinants}, which implies the result on the determinant.
\end{proof}

\subsection{Counting}\label{sec_counting_covers}

Next up, we prove our counting result for regular covers. For simplicity, we restrict to $\SL(m,\FF)$, an example of a sequence of quasisimple groups coming from a simply connected algebraic group. We also restrict to non-arithmetic base surfaces. Our reason for doing this is that the total number of covers grows too slowly for us to apply the argument based on Lubotzky's count of the number of congruence subgroups of an arithmetic group.

\begin{thmrep}{\ref*{thm:many}~(b)}
    Let $X$ be a complete orientable non-arithmetic hyperbolic surface of genus $g$ with $n$ cusps. Then there exist constants $a_X,a_X'>0$ such that $X$ has at least
    \[
        a_X \cdot \frac{\card{\SL(m,\FF)}^{2g+n-3}}{\log(\card{\SL(m,\FF))})}
        \geq
        a_X' \cdot \frac{\card{\FF}^{(m^2-1)\cdot (2g+n-3)}}{(m^2-1)\log(\card{\FF})} 
    \]
pairwise non-isometric $\SL(m,\FF)$-covers whose systoles are at least
\[
(c_{m,X} + \lo(1)) \cdot \log(\card{\SL(m,\FF))}).
\]
for some constant $c_{m,X}>0$ depending on $m$ and $X$ only.
\end{thmrep}

\begin{proof}
    As before, we write $\Gamma \coloneqq \pi_1(X)$ and $G \coloneqq \SL(m, \FF)$. We apply the same strategy as in the proof of Theorem~\ref{thm:many}~(a): we first count the number of distinct subgroups of $\Gamma$ we produce and then divide by an upper bound for the number such subgroups that are conjugate to a given one in $\PSL(2,\RR)$.
    
    Suppose that $\varphi,\psi \in\Hom(\Gamma,G)$ are both surjective and satisfy $\ker(\varphi) = \ker(\psi)$. It then follows from the first isomorphism theorem that there is an automorphism $\sigma \in \Aut(G)$ such that $\psi = \sigma \circ \varphi$. Kantor--Lubotzky \cite{KantorLubotzky} and Liebeck--Shalev \cite[Theorem~1.6]{LiebeckShalev2} proved that $\varphi$ is surjective with high probability, which means that the number of distinct subgroups we produce is
    \[
    (1+\lo(1))\cdot \frac{\card{\Hom(\Gamma,G)}}{\card{\Aut(G)}} \quad \text{as } \card{\FF} \to \infty
    \]
	 The group $\Aut(\SL(m, \FF))\simeq \Aut(\PSL(m, \FF))$ (see \cite{Dieudonne}) and hence its order satisfies 
	 \[
	 \card{\Aut(G)} \leq C\cdot \card{G}\cdot \log(\card{G})
	 \]
    for some universal constant $C>0$ (see for instance \cite[p.\ xvi]{Atlas}). So, using \cite[Corollary~1.3]{LiebeckShalev2} in the closed case, we obtain 
    \[
    (1+\lo(1))\cdot \frac{\card{G}^{2g+n-2}}{C\cdot \log(\card{G})} \quad \text{as } \card{\FF} \to \infty
    \]
    pairwise distinct normal subgroups.
    
    In order to bound how many of our subgroups yield the same isometry class of surfaces, we again use that if $h\in\PSL(2,\RR)$ is such that $H_1 = h H_2 h^{-1}$ for two subgroups $H_1, H_2<\Gamma$ of finite index, then $h\in\mathrm{Comm}_{\PSL(2,\RR)}(\Gamma)$. Because we assume that $\Gamma$ is non-arithmetic, it follows from Margulis's commensurator criterion that $\Gamma<\mathrm{Comm}_{\PSL(2,\RR)}(\Gamma)$ is of finite index. This implies that, given $H<\Gamma$ of index $\card{G}$, 
    \[
    \card{\st{H'<\Gamma}{[\Gamma:H'] = \card{G} \text{ and }\exists h\in\PSL(2,\RR) \text{ s.t. } h Hh^{-1} = H'}} \leq C'_\Gamma \cdot \card{G}.
    \]
    Dividing the number of subgroups by this yields the result.
\end{proof}

\subsection{Symmetric groups}

The bounds we are able to obtain for symmetric groups are significantly weaker than those for finite groups of Lie type, and also than their analogue in the case of graphs \cite{Eberhard}. Let us start with the fact that the systole of a random $\sym_n$-cover tends to infinity. This will be a consequence of a result by Nica \cite[Corollary 1.3]{Nica} in the non-compact case and Magee--Puder \cite[Theorem 1.2]{MageePuder} in the closed case:
\begin{thm}[Magee--Puder, Nica]\label{thm_MNP}
Let $X$ be a hyperbolic surface and $\gamma \in \pi_1(X) \smallsetminus\{e\}$ then
\[
\EE(\mathrm{fix}_n(\gamma)) = \bO(1),\quad \text{as }n\to\infty.
\]
Here $\mathrm{fix}_n(\gamma)$ denote the number of fixed points of $\varphi(\gamma)$ acting on $\{1,\ldots,n\}$, where $\varphi\in\Hom(\pi_1(X),\sym_n)$ is a random homomorphism.
\end{thm}

This leads to the following:
\begin{cor}\label{cor_sys_cover_closed}
We have
\[
\sys(X_{\sym_n}) \to \infty \quad \text{in probability as }n\to\infty.
\]
\end{cor}
\begin{proof} The idea is to combine Theorem~\ref{thm_MNP} with \eqref{eq_unionbound}. By discreteness of the length spectrum of $X$, the sum on the right hand side of \eqref{eq_unionbound} is finite for any fixed $R>0$. Moreover
\[
\PP\Big(\gamma \in \ker(\varphi) \Big) = \PP\Big(\mathrm{fix}_n(\gamma) =n \Big) \leq \frac{\EE(\mathrm{fix}_n(\gamma))}{n},
\]
which tends to $0$ as $n\to\infty$ by the theorem above.
\end{proof}

\begin{rem}\label{rem_symmetric_group}
\begin{itemize}
\item[(a)]
In the closed case, the above can be made effective using a version of \cite[Theorem~1.11]{MageeNaudPuder} that also covers non-primitive elements. Indeed, by a small modification of the proof of that theorem, one can prove that for all primitive $\gamma\in\Gamma$ and all $k\in\NN^*$ such that $\ell(\gamma^k) \leq C\cdot \log(n)$:
\[
\EE(\mathrm{fix}_n(\gamma^k)) = \mathrm{div}(k) + \bO \mathopen{}\left(\frac{\log(n)^A \, \mathrm{div}(k)}{n} \right)\mathclose{}
\]
as $n\to\infty$ \cite{MageeNaudPuder_personal}, where $\mathrm{div}(k)$ denotes the number of divisors of $k$, which in turn is $\lo(k^\eps)$ for all $\eps>0$. Using this in the argument above implies that the systole is at least $(1+\lo(1))\cdot \log\log(\area(X_{\sym_n}))$ with high probability. When one compares with the results and conjectures on random Cayley graphs, it seems unlikely that this is sharp.

\item[(b)] Similar bounds of the order $\log\log(\area)$ can be proved for random $G$-covers of a non-compact surface of finite area by combining
Lemma~\ref{lem_distortion} with \cite[Lemma~2.2]{Eberhard} for symmetric groups and with \cite[Theorem~4]{LiebeckShalev3} for classical groups of growing rank over a fixed finite field.
\end{itemize}
\end{rem}

\section{Conflict of interest statement}

On behalf of all authors, the corresponding author states that there is no conflict of interest.

\section{Data availability statement}

The data for the plot in Figure \ref{fig:examples} is available as an ancillary file at \url{https://arxiv.org/abs/2312.03504}.

\bibliography{bib}
\bibliographystyle{alpha}

\end{document}